\documentclass[11pt]{paper}

\usepackage{amssymb,amsfonts,amsthm,amsmath}
\usepackage{hyperref}
\hypersetup{colorlinks=true,
linktoc=all,linkcolor=black,
citecolor=black, urlcolor=black}
\usepackage{booktabs}

\usepackage{graphicx, verbatim, centernot}
\usepackage{xcolor}

\usepackage{multicol}

\usepackage[a4paper, hmargin=3cm, vmargin={4cm, 4cm}]{geometry}
\usepackage{fancyhdr}
\newtheorem{theorem}{Theorem}[section]
\newtheorem{lemma}[theorem]{Lemma}
\newtheorem{question}{Question}
\newtheorem{proposition}[theorem]{Proposition}

\newtheorem{claim}[theorem]{Claim}
\theoremstyle{definition}
\newtheorem{remark}[theorem]{Remark}

\def\done{{1\hskip-2.5pt{\rm l}}}

\def\La{{\Lambda}}

\def\la{{\lambda}}

\newcommand{\bR}{\mathbb R}
\newcommand{\bC}{\mathbb C}
\newcommand{\bZ}{\mathbb Z}
\newcommand{\bD}{\mathbb D}
\newcommand{\bT}{\mathbb T}

\newcommand{\bP}{\mathbb P}

\newcommand{\bE}{\mathbb E}
\newcommand{\aV}{\mathcal{V}}
\newcommand{\say}[1]{``#1''}

\renewcommand{\limsup}{\mathop{\overline{\lim}}}
\renewcommand{\Re}{\mathrm{Re}\,}

\numberwithin{equation}{section}

\begin{document}

\title{ \vspace{-20pt} The random Weierstrass zeta function I. \\
Existence, uniqueness, fluctuations}
\author{Mikhail Sodin \and
Aron Wennman \and Oren Yakir}

\maketitle

\begin{abstract}
We describe a construction of random meromorphic functions
with prescribed simple poles with unit residues at a given stationary point process.
We characterize those stationary processes with finite second moment for which,
after subtracting the mean, the random function becomes stationary.
These random meromorphic functions can be viewed as random analogues
of the Weierstrass zeta function from the theory of elliptic functions, or
equivalently as electric fields generated by an
infinite random distribution of point charges.
\end{abstract}

\section{Introduction and overview}
\label{s:intro}
Let $\La$ be a stationary random point process in $\bR^d$, $d\ge 2$, and let
$n_\Lambda = \sum_{\la\in\La} \delta_\lambda$ be its counting measure. 
We take the probability space of $\La$ to be
$(\Omega,\mathcal{F},\bP)$, where $\Omega$ is the space of
locally finite configurations in $\bR^d$ and $\mathcal{F}$ is the $\sigma$-algebra
generated by the events
\[
\big\{\La\in\Omega:n_\La(B)=k\big\}, \quad k\in\bZ_{\ge 0},\quad B \text{ Borel subset of } \bR^d.
\] 
Stationarity of $\La$ amounts to invariance of $\bP$ under translations, i.e.\ 
under the maps $T_x:\Omega\to\Omega$, where $T_x\La=\La-x$.
Denote by $c_\Lambda$ the (first) intensity of $\Lambda$, i.e., assume that
$\bE[n_\Lambda]=c_\Lambda m$, where $m$ is the Lebesgue measure, and consider
the following question:

\begin{question}\label{quest1}
For which stationary point processes $\La$ does there exist a stationary random vector field
$V_\La$ with $\operatorname{div} V_\Lambda = n_\La - c_\La m$ in the sense of distributions?
\end{question}

Probably, the first instance of such a field is due to Chandrasekhar, 
who noted in~\cite[Ch.~IV]{Chandra} that, for the Poisson point process $\La$ in $\bR^3$,
the stationary vector field $V_\La$ can be defined by the regularized series
\[
V_\La (x) = \lim_{R\to\infty}\, \sum_{|\la|\le R} \frac{x-\la}{|x-\la|^3} - \kappa c_\La x,
\]
where the summation is over $\la\in\La$ (here and elsewhere,
we skip $\La$ under the summation sign), and $\kappa = 4\pi /3$ is the
volume of the unit ball in $\bR^3$. Chatterjee-Peled-Peres-Romik~\cite[Proposition~1]{CPPR}
gave the rigorous proof of this for the Poisson point process in $\bR^d$ with
$d \ge 3$.

On the other hand, such a stationary field (with a very mild regularity) does not exist
for the planar Poisson process. This follows from Theorem~\ref{thm2} below
but, probably, is not news for experts. Well-studied examples of stationary planar point
processes possessing stationary vector fields are the limiting Ginibre ensemble
and the zero set of the Gaussian Entire Function~\cite{HKPV, NS-WhatIs}.
For the limiting Ginibre ensemble this also follows from Theorem~\ref{thm2} and, likely,
this is known to experts. For the zero set of the Gaussian Entire Function $F(z)$,
the field $V_\Lambda$ written in complex coordinates is nothing but $(F'/F)(z) - \bar z$ 
which is the complex gradient of the stationary potential 
$\log |F(z)|^2 - |z|^2$.
Plausibly, the stationary field exists for two- and three-dimensional Coulomb-type charged 
systems studied by physicists and mathematicians; see the survey 
papers~\cite{Martin,  SerfatySurv1, SerfatySurv2, Lewin} and the references therein.

Since the higher dimensional version of the question does not bring any
essentially new difficulties comparing with the planar 
case\footnote{In $\bC$ we deal with series of simple fractions
$\frac{1}{z-\la}$, $\la\in\La$. Note that, for a two-dimensional vector field $V(z)$
written in complex notation, the divergence is nothing but $\frac1\pi \partial_{\bar z} V(z)$ and that
$\partial_{\bar z}\frac{1}{z-\la}=\pi\delta_\la$. Similarly, in $\bR^d$ with $d\ge 3$, one needs
to deal with series of translates of simple vector fields $\nabla\frac{1}{|x|^{d-2}}=-(d-2)\frac{x}{|x|^{d}}$.},
we will concentrate on the latter. In this case, the question becomes equivalent to the following one:

\begin{question}\label{quest2}
For which stationary planar point processes $\Lambda$ does there exist
a random meromorphic function $f_\Lambda$
with poles exactly at $\Lambda$, all simple with unit residue,
such that the random function $f_\La (z) - \pi c_\La \bar z$
is stationary?
\end{question}

In this paper we will provide an answer to Question~\ref{quest2}
for point processes with a finite second moment, i.e., $\bE[n_\La(B)^2]<\infty$ for
any bounded Borel set $B\subset \bC$. Such stationary point processes admit a 
\emph{spectral measure} $\rho_\La$; see Section~\ref{s:prel} below for the details and examples.
In Section~\ref{s:random-zeta}, we will construct a random analogue $\zeta_\La$
of the Weierstrass zeta function from the theory of elliptic functions.
The function $\zeta_\La$ is meromorphic with poles exactly at $\La$, 
all simple with unit residue, but in general it is not stationary.
One of our findings is Theorem~\ref{thm2} below. In the simplifying case when
the point process $\La$ is ergodic, it states
that the following three conditions are equivalent:
\begin{enumerate}
\item[(i)] The spectral condition
$\displaystyle{\int_{|\xi|\le 1}\frac{{\rm d}\rho_\La(\xi)}{|\xi|^2}<\infty}$
holds.
\item[(ii)] The sum $\displaystyle{\sum_{\la\in\La,\,1\le|\la|\le R}\la^{-1}}$ converges
in $L^2(\Omega,\bP)$ as $R\to\infty$.
\item[(iii)]
For some random constant
$\Psi\in L^2(\Omega,\bP)$ the field
$\zeta_\La(z)-\Psi-\pi c_\La\bar{z}$ is stationary.
\end{enumerate}
Moreover, any solution $f_\La$ to the problem in Question~\ref{quest2} 
with some very mild regularity (e.g., $\bE[|f_\La(0)|]<\infty$) coincides with $\zeta_\La-\Psi$ 
up to a (deterministic) constant, 
so the field in {\rm (iii)} is essentially unique; see Theorem~\ref{thm3} 
and Remark~\ref{remark:ergodic_implies_independence}.
We also note that correcting by the random
constant in condition (iii) is in fact necessary, 
and the natural choice of $\Psi$ is given in Lemma~\ref{lem3}.

The spectral condition {\rm (i)} can be thought of as a \emph{sum rule} for the
two-point measure of $\Lambda$, cf.\ Remark~\ref{rem-cond-kappa} below.

Curiously, if we do not impose any regularity on $f_\La$,
the answer to Question~\ref{quest2} is always positive. To show this one can
use Weiss' construction \cite{Weiss} or a modification of the Krylov-Bogoliubov averaging 
construction for invariant measures \cite{BGLS}. However, when the spectral
condition {\rm (i)} fails, $f_\La$ is necessarily quite \say{exotic} with wild growth at infinity
and very heavy tails, cf.\ \cite{BGLS}.

As an application of the ideas developed here, we study in
\cite{SWY-II} the variance of line integrals of $f_\La$ along dilated
rectifiable curves $R\,\Gamma$ in the large-$R$ limit. When $\Gamma$ encloses a Jordan
domain $\Omega$, this coincides with the \say{charge fluctuation} in $R\,\Omega$,
which is a classical quantity in mathematical physics.
See Remark~\ref{rem-argument} for a further discussion.

Before we proceed, a few words about definitions are in order. 
By a stationary random vector field we mean a measurable map 
$\La\mapsto V_\La$ of $\Omega$ into the space
of (Borel) measurable functions on $\bC$ taking values in the extended complex 
plane $\widehat{\bC}$, such that for all $z\in\bC$,
\[
V_{T_z\La}(\cdot)=V_\La(\cdot + z).
\] 
Equivalently, $V_\La$ is a stationary random vector field if it takes 
the form $V_\Lambda(z)=F(T_z \Lambda)$ 
for some measurable function $F:\Omega\to\widehat{\bC}$.

\paragraph{Related work.}
There is a certain resemblance between the questions studied here
and several well-studied problems.
Among these is the classical question about the growth of the variance of sums and integrals
of stationary processes which boil down to the existence (better to say, non-existence) of
stationary primitives of stationary processes. This was studied by Robinson~\cite{Robinson},
Leonov~\cite{Leonov}, and Ibragimov-Linnik~\cite[Ch XVIII, \S 2, 3]{IbrLinnik} for
stationary processes on $\bZ$ and on $\bR$, and by Davidovich~\cite{Dav}
for stationary processes on $\bZ^d$ and $\bR^d$, $d\ge 2$.
Aizenmann-Goldstein-Lebowitz~\cite[Theorem~3.1]{AGL} gave a criterion for the
existence of a stationary primitive for a stationary point process on $\bR$.
The relevant ergodic theoretic result is Schmidt's
coboundary theorem~\cite[Theorem~11.8]{Schmidt}.

In physics papers, Lebowitz-Martin~\cite{Leb-Martin}
and Alastuey-Jancovici~\cite{Al-Janc}
among other things computed the spectral measure
and the reduced covariance measure for the field and
potential of two- and three-dimensional Coulomb-type systems.

Questions pertaining to existence and uniqueness of
stationary solutions to stochastic PDE (see, for instance,
Vergara-Allard-Desassis~\cite{Vergara})
also belong to this circle of problems.

\paragraph{Wide-sense stationary point processes.}
The main tool in the proofs of the most of our results
will be the spectral theory of generalized point processes,
developed in the 1950-ies by It\^{o}, Gelfand, and Yaglom.
The proofs will not use the translation-invariance of the distribution
of the point process in its full strength, but rely
only on the translation-invariance of the mean and of
the correlations of the point process.
For this reason, with some obvious modifications,
the corresponding results remain valid
for wide-sense stationary point processes $\La$.

\paragraph{Organization.}
The article is organized as follows. In Section~\ref{s:prel}, we discuss the
notion of the spectral measure and some surrounding preliminaries.
Here we mention in some detail the main examples we kept in mind during this work.
In Section~\ref{s:lemmas} we analyze the convergence
properties of reciprocal sums over $\Lambda$, e.g.\
$\sum_{1\le |\lambda-z|\le R}\frac{1}{\lambda^j}$
for $j=1,2$,
as well as their behavior under translations of the center $z$ of summation.
These sums play a central role
in the construction of the random Weierstrass function, which
is carried out in Section~\ref{s:stat-increments}.
This overall scheme works for
any point process, but in general the field obtained only has stationary increments.
In Section~\ref{s:stat} we discuss the existence, uniqueness and covariance structure
of the invariant field $V_\Lambda$ under the above-mentioned spectral condition.
In Section~\ref{s:potential} we conclude with a discussion of
the existence and covariance structure of
random potentials, that is, solutions to the equation
$\Delta \Pi_\La = 2\pi (n_\La - c_\La m)$.

\paragraph{Notation.}
We use the following notation frequently.
\begin{itemize}
\setlength\itemsep{-.3em}
\item $\bC$, $\bR$; the complex plane and the real line
\item $\bD$; the unit disk $\{|z|\le 1\}$
\item $\partial=\partial_z$ and $\bar\partial
=\partial_{\bar z}$; the Wirtinger derivatives
\[
\partial=\frac12\left(\frac{\partial}{\partial x}-{\rm i}\frac{\partial}{\partial y}\right),\qquad
\bar\partial=\frac12\left(\frac{\partial}{\partial x}+{\rm i}\frac{\partial}{\partial y}\right)
\]
\item $m$; the Lebesgue measure on $\bC$
\item $\widehat{f}$; the Fourier transform, with the normalization
\[
\widehat{f}(\xi)=\int_{\bC}\,e^{-2\pi {\rm i}x\cdot \xi}f(x)\,{\rm d}m(x)
\]
\item $\mathfrak D$, $\mathcal S$; the class of compactly supported $C^\infty$-smooth functions
and the class of Schwartz functions, respectively
\item $\bE$, ${\sf Cov}$, ${\sf Var}$; the expectation, covariance and variance (with respect to
an underlying probability space $(\Omega,\mathcal{F},\bP)$)
\item $\mathcal{F}_{\sf{inv}}$; the sigma-algebra of translation-invariant events
\item $T_a$; translation by $a\in\bC$, acting on functions by $T_af(z)=f(z+a)$
and on sets by $T_aS=\{s-a:s\in S\}$
\item $\delta_z$; unit point mass at $z\in\bC$
\item $\rho_\La$; the spectral measure of the point process $\La$
\item $\kappa_\La$, $\tau_\La$; the truncated and reduced
truncated covariance measures for $\La$, respectively.
\end{itemize}
Oftentimes, we will treat sums and series where the summation variable
ranges over a point process $\La$. When this is clear from the context,
we will abuse notation slightly and simply write
\[
\sum_{|\la|\le R}h(\la)= \sum_{\lambda\in \La\cap R\,\bD}h(\la).
\]

We use the standard Landau $O$-notation and the symbol $\lesssim$ interchangeably.
For limiting procedures involving an auxiliary parameter $a$, we use the notation
$f_a(x)=O_a(g(x))$ to indicate that the implicit constant may depend on $a$.

\section{The second-order structure of stationary point processes}
\label{s:prel}

\subsection{The spectral measure}
\label{ss:spectral}
Let $\La$ be a stationary point process in $\bC$ with a finite second moment, that is, we assume that
$\bE[n_\La(B)^2]<\infty$
for any bounded Borel set $B$.
{\em The spectral measure} of $\La$ is a non-negative locally finite measure $\rho_\La$ on
$\bC$ such that the ``Parseval formula'' holds:
\begin{equation}\label{eq21}
{\sf Cov}\bigl[n_\La (\varphi), n_\La(\psi) \bigr]
= \int_{\bC} \widehat \varphi(\xi) \overline{\widehat{\psi}(\xi)}\, {\rm d}\rho_{\La}(\xi)
= \langle \widehat{\varphi}, \widehat{\psi} \rangle_{L^2(\rho_\La)}\,,
\end{equation}
where $\varphi, \psi\in \mathfrak{D}$, $n_\La(\varphi)$ denotes the linear statistic
\[
n_\La (\varphi) = \sum_{\la\in\La} \varphi(\la)\,,
\]
and $\widehat{\varphi}$, $\widehat{\psi}$ are the Fourier transforms, i.e.,
\[
\widehat{\varphi}(\xi) = \int_{\bC} e^{-2\pi {\rm i} \xi \cdot x}\varphi (x)\,
{\rm d}m(x)\, .
\]
Existence of the spectral measure follows from a version of
the Bochner theorem, see Gelfand-Vilenkin~\cite[Ch~III, \S3]{GV}.
In the physics literature the spectral measure is commonly assumed to have a density,
known as the {\em the structure function}.

Similarly, one defines the spectral measure for stationary random measures,
as well as for generalized stationary
random processes (stationary random distributions).

\medskip
It is also worth mentioning that the spectral measures of stationary point processes
(as well as of stationary random measures) are
always {\em translation-bound\-ed}~\cite[Ch.~8]{DV-J}, that is,
for every $r>0$,
$ \sup_{z\in\bC} \rho_\La \bigl( \{\xi\colon |\xi-z|\le r \} \bigr) < \infty $.
Hence, for every $a>2$,
\begin{equation}\label{eq25}
\int_{\bC} \frac{{\rm d}\rho_\La (\xi)}{1+|\xi|^a} < \infty\,.
\end{equation}

\begin{remark}\label{rem:L2-extension}
While the formula \eqref{eq21} initially holds for $\varphi,\psi\in\mathfrak D$,
it readily extends to more general test
functions (and even some tempered distributions) by taking
the closure in $L^2(\rho_\La)$. In fact, the Fourier image of $\mathfrak{D}$
is dense in $L^2(\rho_\La)$.
This follows from the fact that $\mathfrak{D}$ is dense in the space $\mathcal{S}$
of Schwartz functions and that
the Fourier transform is
a topological isomorphism of $\mathcal{S}$. But in $\mathcal{S}$ any
convergent sequence is bounded by $C(1+|\xi|^2)^{-2}$, so applying
the bounded convergence theorem, we find that the relation
\eqref{eq21} holds for any pair of test functions $\varphi,\psi \in \mathcal{S}$.
To see that the Fourier image is dense in the full $L^2$-space,
it is sufficient to show that the closure of $\mathcal{S}$ in $L^2(\rho_\La)$
contains any bounded continuous
function $f\in L^2(\rho_\La)$ with compact support.
But this is again a direct consequence of the translation-boundedness of
$\rho_\La$ and the bounded convergence theorem applied to $f_j=f*h_j$, where 
$h_j(\xi)=j^{2} h(j\xi)$, $h\in\mathcal{S}$, $\int_{\bC} h{\rm d}m=1$, as $j\to\infty$.
\end{remark}

\subsection{The reduced covariance measure}
The spectral measure of a point process is the Fourier transform of ``the reduced
covariance measure'' $\kappa_\La$, which is a signed measure on $\bC$ such that
\begin{align*}
{\sf Cov}\bigl[ n_\La (\varphi), n_\La(\psi) \bigr] &=
\iint_{\bC\times\bC} \varphi(z) \overline{\psi (z')}\, {\rm d}\kappa_\La (z-z')\, {\rm d}m(z) \\
&=\int_\bC \Bigl[ \int_\bC \varphi (z) \overline{\psi (z-w)}\,
{\rm d}m(z) \Bigr]\, {\rm d}\kappa_\La (w)\,,
\end{align*}
see Daley and Vere--Jones~\cite[Ch.~8]{DV-J} (their notation is slightly
different from the one we use here). Recalling that
\[
\bE\bigl[n_\La (\varphi) n_\La(\overline\psi) \bigr] =
\iint_{\bC\times\bC} \varphi(z)\overline{\psi(z')}\, {\rm d}\nu_\La (z-z')\, {\rm d}m(z)
+ c_\La\, \int_\bC \varphi(z)\overline{\psi(z)}\, {\rm d}m(z)\,,
\]
where $c_\Lambda$ is the first intensity of the
point process $\Lambda$ (i.e., the mean number of points
of $\Lambda$ per unit area) and $\nu_\La$ is
the reduced two-point measure of $\La$, we get
that $\kappa_\La = \tau_\La  + c_\La\delta_0$, where
$\tau_\La = \nu_\La  - c_\La^2\, m $ is the
(reduced) truncated two-point measure of $\La$.
Note that in the physics literature it is tacitly assumed that
the measures $\nu_\La$ and $\tau_\La$ have
densities, called the two-point function and truncated
two-point function, respectively.

Similarly, the reduced covariance measure is defined for random
stationary processes and random stationary measures in $\bC$.

\medskip
The total variation of any reduced covariance measure $\kappa_\La$
(and therefore of the reduced truncated measure $\tau_\La$) is
also translation-bounded~\cite[Ch.~8]{DV-J}.

\medskip
It is worth mentioning
that for many point processes the tails
of the measure $\tau_\La$ decay very fast,
which means that on high frequencies the spectral
measure is close to the Lebesgue measure $c_\La m$.
On low frequencies the behavior of the spectral measure
is governed by the Stillinger-Lovett sum rules, which control
the zeroth and the second moments of $\kappa_\La$~\cite{Martin}.

\subsection{The conditional intensity of \texorpdfstring{$\La$}{Lambda}}
We denote by $\mathcal F_{\sf inv} \subset \mathcal F$
the sigma-algebra of translation-invariant events in $\mathcal F$. The random
variable
\begin{equation}
\label{eq:frak-La}
\mathfrak c_\La 
=\pi^{-1}\bE\bigl[ n_\La (\bD) \big| \mathcal F_{\sf inv} \bigr]\,,
\end{equation}
is known as the \emph{conditional intensity} of $\La$. 
Clearly, $\bE[\mathfrak c_\La] = c_\La$, the first intensity
of the point process.
Furthermore, by the ergodic theorem, one can show that
\begin{equation}
\label{eq:ergodic-c-La}
\lim_{R\to\infty}\frac{n_\La(R\,\bD)}{\pi R^2}=\mathfrak{c}_\La
\end{equation}
both almost surely and in $L^2(\Omega,\bP)$ (see \cite[Theorem~12.2.{\rm IV}]{DV-J2}).
It is also not difficult to show (see~\cite[Exercise~12.2.9]{DV-J2}) that 
${\sf Var}[\mathfrak c_\La] = \rho_{\La} (\{0\})$. Putting these pieces together, we arrive at
\begin{equation}
\label{eq:claim1}
{\sf Var}[\mathfrak c_\La] = \rho_{\La} (\{0\})
= \lim_{R\to\infty} {\sf Var}\Bigl[ \frac{n_\La (R\,\bD)}{\pi R^2} \Bigr]\,.
\end{equation}
Hence, the random variable
$\mathfrak c_\La$ does not degenerate to the deterministic intensity $c_\La$
if and only if $n_\La (R\,\bD)$ asymptotically has the variance
of the maximal possible order $R^4$, i.e., ``hyper-fluctuates'', and this in turn
is equivalent to $\rho_\La(\{0\})>0$.
Note that if the point process $\La$ is ergodic,
then the sigma-algebra $\mathcal F_{\sf inv}$ contains
only events of probability $0$ or $1$,
and therefore,  $\mathfrak c_\La $ is constant.

The simplest example of a point process with
a spectral measure having an atom at the origin
is a random mixture of two independent
Poisson processes having different intensities
(such processes are called Cox point processes).
For a more general construction, take $\La$ to be any ergodic point process with spectral
measure $\rho_\La$, and denote by $L$
a positive non-degenerate random variable with ${\sf Var}[L]<\infty$. Then
$\La'=L^{-\frac12}\La$ is a point process
with finite second moment. 
In view of \eqref{eq:ergodic-c-La}, we get that
\[
\mathfrak{c}_{\La'}
=\lim_{R\to\infty}\frac{n_{\La'}(R\,\bD)}{\pi R^2}=
\lim_{R\to\infty}\frac{n_{\La}(L^{\frac12}R\,\bD)}{\pi R^2}
=c_\La L,
\]
and hence \eqref{eq:claim1} gives that $\rho_{\La'}(\{0\})=c_\La^2\,{\sf Var}[L]$.

\subsection{Examples}
\label{s:examples}
Here, we will make a short stop to present several examples of
stationary two-dimensional point processes, which
we kept in mind starting this work. For all these examples,
the spectral measure can be computed without much effort.

\paragraph{The Poisson point process.}
In this case, the two-point function identically equals $c_\La^2$ ($c_\La$ is
the intensity of the Poisson process), the truncated two-point function vanishes,
$\rho_\La = c_\La^2 m$, and~\eqref{eq21} is nothing but
the classical Parseval-Plancherel formula.

\paragraph{The limiting Ginibre ensemble.} This is the large $N$ limit of
the eigenvalues of the Ginibre ensemble of $N\times N$ random matrices with
independent standard complex Gaussian entries. One important feature of
the limiting ensemble is its determinantality, which means that its
$k$-point functions can be expressed in terms of the determinants
\[
r(z_1, \ldots , z_k) = \pi^{-k} e^{-\sum_{i=1}^k |z_i|^2}
\det (e^{z_i\bar z_j} )_{1\le i, j \le k}\,,
\]
see, for instance,~\cite[Section~4.3.7]{HKPV}.
This immediately yields the simple expression $-\pi^{-2}e^{-\pi |z|^2}$
for the truncated two-point function and that $c_\La=\pi^{-1}$, which, in turn, yields
that the spectral measure is absolutely continuous with the density
$\pi^{-1}(1-e^{-\pi |\xi|^2})$.

\paragraph{Zeroes of the Gaussian Entire Function.}
The Gaussian Entire Function (GEF, for short) is defined by the random Taylor series
\[
F(z) = \sum_{n\ge 0} \zeta_n \frac{z^n}{\sqrt{n!}}
\]
with independent standard complex Gaussian
coefficients $\zeta_n$. The most basic facts about
GEFs and their zeroes can be found in~\cite{HKPV, NS-WhatIs}.
The two-point function and the spectral measure of the zero point process
were explicitly computed by Forrester-Honner~\cite{FH}
and Nazarov-Sodin~\cite{NS-IMRN}. The intensity is given by $c_\La=\pi^{-1}$, 
the truncated two-point function
equals $h(|z|)$, where
\[
h(r)=\frac12\, \frac{{\rm d}^2}{{\rm d}r^2}\, r^2 (\coth r - 1)\,,
\]
while the
spectral measure is absolutely continuous with the density
\[
\pi^3 |\xi|^4\, \sum_{\ell\ge 1} \frac1{\ell^3}\, e^{-\pi^2|\xi|^2/\ell}\,.
\]

\paragraph{``Stationarized'' random Gaussian perturbation of the lattice.}
This is a stationary point process defined as
$ \La^a = \bigl\{\nu + \zeta^a_\nu +  U\bigr\}_{\nu\in\bZ^2} $,
where $\zeta^a_\nu$ are independent complex-valued
Gaussian random variables with the variance $a>0$,
and $U$ is uniformly distributed on $[0, 1]^2$ and is independent of all $\zeta^a_\nu$s.
In this case, the spectral measure is also not difficult to compute
(see, for instance, Yakir~\cite[\S 3]{Yakir}). It is a sum of an absolutely
continuous measure, which is similar to the one of the limiting Ginibre ensemble,
and a discrete measure with atoms at $\bZ^2\setminus\{0\}$,
\[
\rho_{\La^a}=(1-e^{-2a\pi^2|\xi|^2})m
+\sum_{\nu\in\bZ^2\setminus\{0\}} e^{-2a\pi^2|\nu|^2} \delta_\nu\,.
\]
Moreover, the reduced covariance measure $\kappa_{\La^a}$ is given by
\[
\kappa_{\La^a}=\delta_0-m+\sum_{\nu\ne 0}(2a\pi)^{-1}e^{-|z-\nu|^2/(2a)}m.
\]
This can be obtained by a direct computation on the spatial side,
or by taking the inverse Fourier transform of $\rho_{\La^a}$.
In the limit as $a\to 0$, we obtain the randomly shifted lattice
$\La = \bigl\{\nu + U\bigr\}_{\nu\in\bZ^2}$ with the purely atomic spectral measure
$\rho_\La = \sum_{\nu\in\bZ^2\setminus\{0\}} \delta_\nu $.

\section{Reciprocal sums over stationary point processes}
\label{s:lemmas}
\subsection{Convergence of three series}
Recall the notation $(\Omega, \mathcal F, \bP)$ for the probability
space on which the point process $\La$ is defined.
To define the random Weierstrass zeta function,
we need three lemmas.

\begin{lemma}\label{lem1}
Let $\La$ be a stationary point process in $\bC$ having a 
finite second moment. Then almost surely and in
$L^2(\Omega, \bP)$,
\[
\sum_{|\la|\ge 1} \frac1{|\la|^3} < \infty\,.
\]
\end{lemma}

For $R>1$, we set
\begin{equation}
\label{eq:def-Psi-ell}
\Psi_\ell (R) = \sum_{1\le |\la|\le R} \frac1{\la^\ell}\,, \qquad \ell = 1, 2\,.
\end{equation}
The behavior of these two sums as $R\to\infty$ will be important for us.

\begin{lemma}\label{lem2}
Let $\La$ be a stationary point process in $\bC$ having a finite second moment. Then
there exists a random variable $\Psi_2(\infty)\in L^2(\Omega, \bP)$ such that
\[
\lim_{R\to\infty} \bE\bigl[ |\Psi_2(R)- \Psi_2(\infty) |^2 \bigr] = 0\,.
\]
\end{lemma}

The convergence of $\Psi_1(R)$ in $L^2(\Omega, \bP)$ requires an additional property
of the spectral measure $\rho_\La$ of $\La$.

\begin{lemma}\label{lem3}
Let $\La$ be a stationary point process in $\bC$ with spectral measure $\rho_\La$,
and assume that
\[
\int_{0<|\xi|\le 1}\frac{{\rm d}\rho_\La(\xi)}{|\xi|^2}<\infty.
\]
Then there exists a random variable $\Psi_1(\infty)\in L^2(\Omega, \bP)$ such that
\[
\lim_{R\to\infty} \bE\bigl[ |\Psi_1(R)- \Psi_1(\infty) |^2 \bigr] = 0\,.
\]
\end{lemma}

It is worth mentioning that the conditional convergence of the
series $\sum_\La \la^{-\ell}$ with $\ell = 1, 2$, when $\La$ is the
limiting Ginibre process, or the zero process of GEF appear as auxiliary
results in Ghosh-Peres~\cite[Sections~8 and 10]{GP}.

\subsection{Proof of the three lemmas}
The \emph{Bessel function of order }$\nu\in \bZ_+$ is given by
\begin{equation*}
J_\nu(x) = \frac{1}{2\pi} \int_{-\pi}^{\pi}
e^{{\rm i} (x\sin \theta -\nu \theta)} {\rm d} \theta.
\end{equation*}
We will frequently use two basic properties
of the Bessel function, namely the asymptotic
formulas
\begin{equation}\label{eq:bessel_function_near_zero}
J_\nu(x) = \frac{1}{\nu!} \left(\frac{x}{2}\right)^\nu
+ o(x^\nu),\qquad \text{as }\, x\to 0,
\end{equation}
and
\begin{equation}\label{eq:bessel_function_at_infinity}
\sup_{|x|\ge 1} |x|^{3/2}\left|J_\nu(x) - \sqrt{\frac{2}{\pi x}}
\cos\left(x - \frac{\nu\pi}{2} - \frac{\pi}{4}\right) \right| < \infty.
\end{equation}
For the proof of both facts see, for instance, \cite[Ch.7]{watson}.
\subsubsection{Proof of Lemma \ref{lem1}}
Denote by $\widetilde{\Psi}_3(R) = \sum_{1\leq |\la| \leq R} |\la|^{-3}$.
We will show that the limit $\widetilde{\Psi}_3(\infty)$ exists both
almost surely and in $L^2(\Omega,\bP)$. We have
\begin{align*}
\bE[\widetilde{\Psi}_3(R)] &= \int_{\{1\leq |x|\le R\}}
\frac{\bE[{\rm d}n_\La(x)]}{|x|^3} \\ &=
\bE[\mathfrak{c}_\La] \int_{\{1\leq |x|\le R\}}
\frac{{\rm d}m(x)}{|x|^3} \lesssim \int_{1}^{\infty}
\frac{{\rm d}s}{s^2} < \infty\, ,
\end{align*}
which implies that $\widetilde{\Psi}_3(R)$ converges almost surely,
as the sum consists of positive terms.

To show that $\widetilde{\Psi}_3(\infty) \in L^2(\Omega,\bP)$,
we use the Parseval identity
\eqref{eq21} to move to the spectral side, which gives
\begin{align*}
{\sf Var}\left[\widetilde{\Psi}_3(R)\right]
&= \int_{\bC} \left|\int_{\{1\leq |x|\leq R\}}e^{-2\pi {\rm i} \xi \cdot x}
\frac{{\rm d} m(x)}{|x|^3}\right|^2 {\rm d} \rho_{\La}(\xi) \\
&= \int_{\bC} \left|\int_{1}^{R}\left(\int_{-\pi}^{\pi}e^{-2\pi {\rm i}
|\xi|t \cos \theta} {\rm d}\theta \right) \frac{{\rm d}t}{t^2}\right|^2 
{\rm d} \rho_{\La}(\xi) \\
&= \int_{\bC}\left|2\pi\int_{1}^{R}
\frac{J_0(2\pi |\xi|t)}{t^2} {\rm d}t\right|^2 {\rm d} \rho_{\La}(\xi).
\end{align*}
The integrand above is bounded uniformly in $R$. Thus, we need to check
how fast it decays as $|\xi|$ becomes large. For this, we use the
asymptotic formula \eqref{eq:bessel_function_at_infinity} for the Bessel function
and see that, for $|\xi|\ge 1$,
\begin{align*}
\left|\int_{1}^{R} \frac{J_0(2\pi |\xi|t)}{t^2} {\rm d}t\right|
& \lesssim \frac{1}{|\xi|^{3/2}} \int_{1}^{R} \frac{{\rm d}t}{t^{7/2}}
+ \frac{1}{|\xi|^{1/2}}\left|\int_{1}^{R}\frac{\cos\left(2\pi|\xi|t
-\frac{\pi}{4}\right)}{t^{5/2}} \, {\rm d}t\right| \\ 
& \lesssim \frac{1}{|\xi|^{3/2}}
+ \frac{1}{|\xi|^{3/2}}\left[1+ \frac{1}{R^{5/2}}
+ \int_{1}^{R}\frac{{\rm d}t}{t^{7/2}} \right].
\end{align*}
Therefore,
\[
\sup_{R\ge 1} \left|\int_{1}^{R} \frac{J_0(2\pi |\xi|t)}{t^2}
{\rm d}t\right|^2 \lesssim \min\{1, |\xi|^{-3}\}
\]
and the function on the RHS is ${\rm d}\rho_{\La}$-integrable.
Hence, we can apply the dominated convergence theorem and deduce that
$\widetilde{\Psi}_3(R)$ converge in $L^2(\Omega,\bP)$ as well.
\qed
\subsubsection{Proof of Lemma \ref{lem2}}
Lemma~\ref{lem2} states that the random variables
$\Psi_2(R) = \sum_{1\leq |\la|\leq R} \la^{-2}$
converge in $L^2(\Omega,\bP)$ as $R\to \infty$ to a limiting random
variable $\Psi_2(\infty) \in L^2(\Omega,\bP)$. This will follow once we
show that $\{\Psi_2(R)\}_{R\ge1}$ satisfies
the Cauchy criterion in $L^2(\Omega,\bP)$,
which we do by a computation. For any $R\ge 1$,
\[
\bE\left[\Psi_2(R)\right] = c_\La
\int_{\{1\leq |\la|\leq R\}} \frac{{\rm d}m(\la)}{\la^2} = 0\, ,
\]
and by the Parseval formula \eqref{eq21}, applied with
$\varphi(x)= \displaystyle{x^{-2}}\done_{\{R\leq |x|\leq R^\prime\}}$
for any $R^\prime> R$, we have
\begin{align*}
\bE\left[\left|\Psi_2(R^\prime) - \Psi_2(R)\right|^2\right]
&= \int_{\bC} \left|\int_{\{R \leq |x|\leq R^\prime\}}e^{-2\pi {\rm i} \xi \cdot x}
\frac{{\rm d} m(x)}{x^2} \right|^2 {\rm d} \rho_{\La}(\xi) \\
&= \int_{\bC\setminus\{0\}} \left|\int_{\{R \leq |x|\leq R^\prime\}}e^{-2\pi {\rm i}
\xi \cdot x} \frac{{\rm d} m(x)}{x^2} \right|^2 {\rm d} \rho_{\La}(\xi).
\end{align*}
We may rewrite the inner integral on the RHS as
\begin{align*}
\int_{\{R \leq |x|\leq R^\prime\}}e^{-2\pi {\rm i} \xi \cdot x}
\frac{{\rm d} m(x)}{x^2} &= e^{-2{\rm i}\chi}\int_{R}^{R^\prime}
\left(\int_{-\pi}^{\pi}e^{-2\pi {\rm i} |\xi|r \cos\theta }
e^{-2{\rm i}\theta}{\rm d}\theta\right) \frac{{\rm d} r}{r} \\
& \qquad = -2\pi e^{-2{\rm i}\chi} \int_{R}^{R^\prime}
\frac{J_2(2\pi |\xi| r)}{r} \, {\rm d}r \\
& \qquad \qquad= e^{-2{\rm i}\chi}
\left(\frac{J_1(2\pi R^\prime|\xi|)}{R^\prime|\xi|}
- \frac{J_1(2\pi R|\xi|)}{R|\xi|}\right)\, ,
\end{align*}
where $\chi=\arg \xi$, and where in the last equality we used
that $(J_1(x)/x)^\prime = -J_2(x)/x$. Thus,
\begin{align*}
\bE\left[\left|\Psi_2(R^\prime) - \Psi_2(R)\right|^2\right]
&= \int_{\bC\setminus\{0\}} \left|\frac{J_1(2\pi R^\prime|\xi|)}{R^\prime|\xi|}
- \frac{J_1(2\pi R|\xi|)}{R|\xi|}\right|^2 {\rm d} \rho_{\La}(\xi).
\end{align*}
By the near-origin asymptotics
\eqref{eq:bessel_function_near_zero} of the Bessel functions,
$J_1(x)/x$ is
bounded near the origin, and, together with the asymptotic formula
\eqref{eq:bessel_function_at_infinity} and
the translation-boundedness of $\rho_\La$, we get
\begin{align*}
\lim_{R\to\infty}\sup_{R^{\prime}\ge R} 
& \bE\left[\left|\Psi_2(R^\prime)
- \Psi_2(R)\right|^2\right] \\ &\lesssim \limsup_{R\to\infty}
\left[\rho_{\La}(\{0<|\xi|\leq R^{-1/2}\}) + \int_{\{|\xi|\ge R^{-1/2}\}}
\frac{{\rm d} \rho_{\La}(\xi)}{(1+R|\xi|)^3}\right] = 0.
\end{align*}
That is, $\Psi_2(R)$ is a Cauchy sequence in $L^2(\Omega,\bP)$.
\qed

\subsubsection{Proof of Lemma~\ref{lem3}}
We start by computing $\bE[|\Psi_1(R^\prime) - \Psi_1(R)|^2]$
for $R^\prime>R$. For any $R\ge 1$ we have
\[
\bE\left[\Psi_1(R)\right]=c_\La \int_{\{1\leq |\la|\leq R\}} 
\frac{{\rm d} m(\la)}{\la} = 0
\]
and so, for any $R^\prime>R$, the Parseval formula \eqref{eq21} gives that
\begin{align*}
\bE\left[\left|\Psi_1(R^\prime) - \Psi_1(R)\right|^2\right]
&= \int_{\bC}\left|\int_{\{R \leq |x|\leq R^\prime\}}e^{-2\pi {\rm i} \xi \cdot x}
\frac{{\rm d} m(x)}{x}\right|^2 {\rm d} \rho_\La(\xi) \\
&= \int_{\bC \setminus\{0\}}
\left|\int_{\{R \leq |x|\leq R^\prime\}}e^{-2\pi {\rm i} \xi \cdot x}
\frac{{\rm d} m(x)}{x}\right|^2 {\rm d} \rho_\La(\xi).
\end{align*}
Since
\[
e^{-2\pi {\rm i} \xi \cdot x}
=e^{-\pi {\rm i} (\xi\bar{x} + \bar{\xi} x)} = \bar\partial_{x}
\left(\frac{e^{-2\pi {\rm i} \xi \cdot x}}{-\pi {\rm i} \xi}\right),
\]
we can use the Cauchy-Green formula to obtain
\begin{align*}
\int_{\{R \leq |x|\leq R^\prime\}}e^{-2\pi {\rm i} \xi \cdot x} \frac{{\rm d} m(x)}{x}
&= \frac{1}{2\pi \xi} \left(\int_{\{|x|=R^\prime\}}
e^{-2\pi {\rm i} \xi \cdot x} \frac{{\rm d}x}{x}
-\int_{\{|x|=R\}}e^{-2\pi {\rm i} \xi \cdot x} \frac{{\rm d}x}{x}\right)
\\ & = \frac{\bar{\xi}}{|\xi|} \frac{J_0(2\pi R^\prime|\xi|)-J_0(2\pi R|\xi|)}{\xi},
\end{align*}
and plugging this into the above formula for
$\bE\left[\left|\Psi_1(R^\prime) - \Psi_1(R)\right|^2\right]$ gives
\begin{equation}{\label{eq:formula_variance_psi_1}}
\bE\left[\left|\Psi_1(R^\prime) - \Psi_1(R)\right|^2\right]
= \int_{\bC \setminus\{0\}} \left(J_0(2\pi R^\prime|\xi|)
-J_0(2\pi R|\xi|)\right)^2 \frac{{\rm d}\rho_{\La}(\xi)}{|\xi|^2}.
\end{equation}
Since $J_0$ is bounded, we can use \eqref{eq:formula_variance_psi_1}
and the asymptotic formula \eqref{eq:bessel_function_at_infinity} for $J_0$ to get
\begin{align*}
\lim_{R\to\infty}\sup_{R^{\prime}\ge R}
&\bE\left[\left|\Psi_1(R^\prime) - \Psi_1(R)\right|^2\right]
\\ & \lesssim \limsup_{R\to \infty}
\left[\int_{\{0<|\xi|\leq R^{-1/4}\}} \frac{{\rm d}\rho_{\La}(\xi)}{|\xi|^2}
+ \frac{1}{R} \int_{\{|\xi|\ge R^{-1/4}\}} \frac{{\rm d}\rho_{\La}(\xi)}{|\xi|^3}\right] \\
& \qquad \lesssim \limsup_{R\to\infty}\left[R^{-3/4}\int_{\{0<|\xi|\leq 1\}}
\frac{{\rm d}\rho_{\La}(\xi)}{|\xi|^2} + \frac{1}{R}\int_{\{|\xi|\ge 1\}}
\frac{{\rm d}\rho_{\La}(\xi)}{|\xi|^3}\right] = 0.
\end{align*}
That is, $\{\Psi_1(R)\}_{R\geq 1}$ is Cauchy in
$L^2(\Omega,\bP)$. This completes the proof.
\qed

\subsection{Translation properties of reciprocal sums}
The random variables $\Psi_\ell(R)$ are 
defined by summation over annuli centered
at the origin. It will be essential to understand
the effect of translating the center of summation
in these sums. This amounts to understanding
the summation over the \emph{lunar domains}
formed as the symmetric difference of two large disks with different centers.
\begin{lemma}\label{lemma:lunar_domains}
Let $\La$ be a stationary point process in $\bC$ with finite second moment,
with conditional intensity
$\mathfrak{c}_\La = \pi^{-1}\bE[n_\La({\bD}) \mid \mathcal{F}_{{\sf inv}}]$.
Then, for any $u,v,z\in \bC$, we have
\[
\lim_{R\to \infty} {\sf Var}\left[\sum_{|\la-u|\leq R} \frac{1}{z-\la}
- \sum_{|\la-v|\leq R} \frac{1}{z-\la}
+ \pi \mathfrak{c}_\La (\overline{u-v})\right] = 0.
\]
\end{lemma}
Although we will need this in the paper, 
we remark that the convergence is locally uniform in $u,v,z$.

In the case when $\La$ is the Poisson point process, Lemma~\ref{lemma:lunar_domains}
was proved by Chatterjee, Peled, Peres and Romik in \cite[Lemma~8]{CPPR}
where they obtain the analogous result
for $d\ge3$, but the same proof works also for $d=2$.

\begin{proof}
By stationarity of $\La$, it suffices to prove the lemma for $z=0$.
We will assume that $R$ is large enough so that
both $u$ and $v$ are contained inside $R\bD$.
By the Cauchy-Green formula it holds that
$\int_{\{|x-u|\le R\}} \frac{{\rm d}m(x)}{x} = \pi \bar u$, and thus
\[
\bE \Big[\sum_{|\la-u|\leq R} \frac{1}{\la}\Big]
= \bE[\mathfrak{c}_\La] \left(\int_{\{|x-u|\leq R\}} \frac{{\rm d}m(x)}{x}\right)
= \pi c_\La \, \bar u.
\]
Introduce the notation
\begin{align*}
X &= \sum_{|\la-u|\leq R} \frac{1}{\la} - \sum_{|\la-v|\leq R} \frac{1}{\la}
- \pi c_\La (\overline{u-v})\, , \quad \text{and} \ \
Y= \pi \big(\mathfrak{c}_{\La} -  c_\La\big).
\end{align*}
Clearly, $\bE[X] = 0$ and $\bE[X\,|\, Y] = Y(\overline{u-v})$, since the
law of the point process $\La$, conditioned
on $Y$, has the intensity $\mathfrak{c}_\La$. Thus,
\begin{align*}
{\sf Var}\left[X - Y(\overline{u-v})\right]
&= \bE\left[\left|X-Y(\overline{u-v})\right|^2\right] \\
&= \bE\left[\left|X-\bE[X\mid Y]\right|^2\right] \\
&= \bE\left|X\right|^2 - \bE\left|Y(\overline{u-v})\right|^2
\qquad \text{(Pythagoras' theorem)} \\
&= {\sf Var}[X] - \pi^2|u-v|^2\, {\sf Var}(\mathfrak{c}_{\La}) \\
&= {\sf Var}[X] - \pi^2|u-v|^2\rho_{\La}(\{0\}),
\qquad \text{(by \eqref{eq:claim1})}.
\end{align*}
Plugging in the definition of $X$ and $Y$ yields that
\begin{align}\label{eq:lunar_domain_variance_split}
\begin{split}
{\sf Var}&\left[\sum_{|\la-u|\leq R} \frac{1}{\la}
- \sum_{|\la-v|\leq R} \frac{1}{\la} -
\pi \mathfrak{c}_\La (\overline{u-v})\right] \\
&= \int_{\bC}\left|\left(\int_{\{|x-u|\leq R\}}
- \int_{\{|x-v|\leq R\}}\right)e^{-2\pi {\rm i} \xi \cdot x}
\frac{{\rm d} m(x)}{x} \right|^2 {\rm d} \rho_{\La}(\xi)
- \pi^2 |u-v|^2 \rho_{\La}(\{0\}) \\
&= \int_{\bC\setminus\{0\}}\left|\left(\int_{\{|x-u|\leq R\}}
- \int_{\{|x-v|\leq R\}}\right)e^{-2\pi {\rm i} \xi \cdot x}
\frac{{\rm d} m(x)}{x} \right|^2 {\rm d} \rho_{\La}(\xi) = I_1 + I_2 \, ,
\end{split}
\end{align}
where,
\begin{align*}
I_1 &\stackrel{{\rm def}}{=} \int_{\{0<|\xi|\leq R^{-1/4}\}}
\left|\left(\int_{\{|x-u|\leq R\}}
- \int_{\{|x-v|\leq R\}}\right)e^{-2\pi {\rm i} \xi \cdot x}
\frac{{\rm d} m(x)}{x} \right|^2 {\rm d}\rho_{\La}(\xi)\, ,
\\ I_2 &\stackrel{{\rm def}}{=} \int_{\{|\xi|\geq R^{-1/4}\}}
\left|\left(\int_{\{|x-u|\leq R\}}- \int_{\{|x-v|\leq R\}}\right)
e^{-2\pi {\rm i} \xi \cdot x} \frac{{\rm d} m(x)}{x} \right|^2 {\rm d}\rho_{\La}(\xi)\, .
\end{align*}
Since the integrand is bounded by some constant
$C=C(u,v)>0$ (independent of $R$), we can bound $I_1$ as
\[
I_1 \lesssim_{u,v} \rho_{\La}\left(\{0<|\xi|\leq R^{-1/4}\}\right)
\xrightarrow{R\to \infty} 0.
\]
To bound the second integral $I_2$, we use the Cauchy-Green
formula to compute the inner integral:
\begin{align*}
\int_{\{|x-u|\leq R\}} e^{-2\pi {\rm i} \xi \cdot x} \frac{{\rm d} m(x)}{x}
&= \frac{1}{-{\rm i}\pi \xi} \int_{\{|x-u|\leq R\}}
\frac{\bar\partial_{x}\left(e^{-2\pi {\rm i} \xi \cdot x}\right)}{x} {\rm d} m(x)
\\ &= \frac{1}{-{\rm i}\xi} \left(\frac{1}{2\pi {\rm i}}
\int_{\{|z-u|= R\}} e^{-2\pi {\rm i} \xi \cdot z}\frac{{\rm d} z}{z} - 1\right)
\\ &= \frac{1}{2\pi \xi} \int_{\{|z-u|= R\}} 
e^{-2\pi {\rm i} \xi \cdot z}\frac{{\rm d} z}{z}
- \frac{1}{{\rm i}\xi}.
\end{align*}
Plugging back the above in the definition of $I_2$ gives us
\[
I_2 = \int_{\{|\xi|\geq R^{-1/4}\}} \left|\frac{g_u^R(\xi)
- g_v^R(\xi)}{2\pi \xi}\right|^2 {\rm d}\rho_{\La}(\xi),
\]
where
\[
g_u^{R}(\xi) \stackrel{{\rm def}}{=} 
\int_{\{|z-u| = R\}}e^{-2\pi {\rm i} \xi \cdot z}
\frac{{\rm d} z}{z} = e^{-2\pi {\rm i} \xi \cdot u}
\int_{0}^{2\pi} e^{-2\pi {\rm i} |\xi|R \cos \theta}
\frac{Re^{{\rm i}\theta}}{Re^{{\rm i}\theta} + u} {\rm d} \theta.
\]
By the standard stationary phase bound
(see Proposition~\ref{prop:stat-phase} in the Appendix),
for any $u\in \bC$ we have 
$\left|g_u^R(\xi)\right| \lesssim_u (1+R|\xi|)^{-1/2}$.
Hence, we see that
\begin{align*}
I_2 \lesssim_{u,v} \int_{\{|\xi|\geq R^{-1/4}\}}
&\frac{{\rm d}\rho_{\La}(\xi)}{|\xi|^2(1+R|\xi|)}
\\ & \lesssim \frac{1}{\sqrt{R}} \, \rho_{\La}(\bD) 
+ \frac{1}{R}\int_{\{|\xi| \ge 1\}}
\frac{{\rm d}\rho_{\La}(\xi)}{|\xi|^3} \xrightarrow{R\to \infty} 0.
\end{align*}
Plugging back the bounds on $I_1$ and $I_2$ into
\eqref{eq:lunar_domain_variance_split}, we get that
\[
\lim_{R\to \infty}{\sf Var}\left[\sum_{|\la-u|\leq R} \frac{1}{\la}
- \sum_{|\la-v|\leq R} \frac{1}{\la} 
- \pi \mathfrak{c}_\La (\overline{u-v})\right]
\leq \limsup_{R\to \infty} \left(I_1 + I_2\right) = 0
\]
which gives the lemma.
\end{proof}

\section{Fields and potentials with stationary increments}
\label{s:stat-increments}

\subsection{The Weierstrass zeta function}
There is an evident analogy with
the classical Weierstrass zeta function 
from the theory of elliptic functions.
Suppose for a moment that $\La$ is
a non-degenerate lattice in $\bC$, then
\[
\zeta_\La (z) = \frac1{z} + \sum_{\la\in\La\setminus\{0\}}
\Bigl( \frac1{z-\la} + \frac1{\la} + \frac{z}{\la^2} \Bigr)\,.
\]
In our context, it is more convenient to use a
different normalization, which goes back to Eisenstein.
Letting
\begin{align*}
\Psi_2 &= \lim_{R\to\infty}\, \sum_{0<|\la|\le R} \frac1{\la^2}\,, \\
\zeta_\La (z) &= \frac1{z} + \sum_{\la\in\La\setminus\{0\}}
\Bigl( \frac1{z-\la} + \frac1{\la} + \frac{z}{\la^2} \Bigr) - \Psi_2 z\,,
\end{align*}
and noting that, for each $R$,
\[
\sum_{0<|\la|\le R} \frac1{\la} = 0,
\]
we find that
\[
\zeta_\La (z)= \lim_{R\to\infty} \sum_{|\la|\le R} \frac1{z-\la}\,.
\]

Let $c_\La$ be the inverse area of the fundamental domain of $\La$. Then,
it is not difficult to show (see Taylor~\cite[Appendix]{Taylor}) that 
{\em the function $\zeta_\La(z) -\pi c_\La\bar z$ is $\La$-invariant}. 
Note that in this case we have the limiting relation
\[
c_\La = \lim_{R\to\infty} \frac{n_\La (R\,\bD)}{\pi R^2}.
\]

It is worth to mention that in~\cite{Taylor} Taylor computed
the Fourier expansions of the functions $\zeta_\La (z+a)-\zeta_\La (z)$ and
$\zeta_\La(z) -\pi c_\La\bar z$.

\subsection{The random Weierstrass zeta function}
\label{s:random-zeta}
We return to the probabilistic setting, and let $\La$ be a stationary
point process with finite second moment, and recall the quantities 
$\Psi_1(R)$ and $\Psi_2(R)$ 
defined in \eqref{eq:def-Psi-ell}. 
Lemmas \ref{lem1} and \ref{lem2} in Section~\ref{s:lemmas}
allow us to define the random meromorphic function
\begin{equation}\label{eq20}
\zeta_\Lambda (z) \stackrel{\rm def}= \sum_{|\la|<1} \frac1{z-\la}
+ \sum_{|\la|\ge 1} \Bigl( \frac1{z-\la}
+ \frac1\la + \frac{z}{\la^2} \Bigr) - \Psi_2(\infty) z\,,
\end{equation}
where $\Psi_2(\infty)=\lim_{R\to\infty}\Psi_2(R)$.
Note that, for any $R>1$,
\begin{equation}\label{eq27}
\zeta_\La (z) = \sum_{|\la|\le R} \frac1{z-\la}
+ \Psi_1(R) + \bigl( \Psi_2(R) - \Psi_2(\infty) \bigr) z
+ \sum_{|\la|> R} \frac{z^2}{\la^2(z-\la)}\,.
\end{equation}
Again by Lemmas~\ref{lem1} and \ref{lem2}, the last two terms on
the right-hand side of \eqref{eq27} tend to zero as $R\to\infty$,
where the convergence is in $L^2(\Omega, \bP)$ and is locally uniform in $z$.
This hints that the function $\zeta_\La$ is not too far from being a stationary one.

\begin{theorem}\label{thm1}
Let $\La$ be a stationary point process in $\bC$ having a finite second moment.
Then the random meromorphic function $\zeta_\La$, as defined in~\eqref{eq20},
has stationary increments. That is, for any $a\in\bC$,
the distribution of the random meromorphic function
${\sf\Delta}_a \zeta_\La (z) = \zeta_\La(z+a)-\zeta_\La(z)$ is stationary.
\end{theorem}

\begin{proof}
We have
\begin{equation*}
\zeta_\La (z)  \\
= \sum_{|\la|\le R} \frac1{z-\la} + \Psi_1(R) +
(\Psi_2(R)-\Psi_2(\infty))z + \chi_R(\La,z)z^2\,,
\end{equation*}
where
\[
\chi_R(\La,z) = \sum_{|\la|> R} \frac{1}{\la^2(z-\la)}.
\]
Hence, by Lemmas~\ref{lem1} and \ref{lem2},
\begin{equation*}
\lim_{R\to \infty} \bE \, \bigg|\zeta_\La (z)
- \sum_{|\la|\le R} \frac{1}{z-\la} - \Psi_1(R)\bigg|^2 = 0\, ,
\end{equation*}
for any $z\in \bC$ fixed. The above, 
together with the \say{lunar lemma}
(Lemma~\ref{lemma:lunar_domains}), gives us
\begin{equation}
\label{eq:pointwise_L_2_limit_for_zeta}
\lim_{R\to \infty} \bE \, \bigg|\zeta_\La (z)
- \sum_{|\la-z|\le R} \frac{1}{z-\la}
- \pi \mathfrak{c}_\La \bar z - \Psi_1(R) \bigg|^2 = 0 \, ,
\end{equation}
and therefore,
\[
\lim_{R\to \infty} \bE\, \bigg|\zeta_{\La}(z+a) - \zeta_{\La}(z)
- \sum_{|\la-z-a|\le R} \frac{1}{z+a-\la}
+ \sum_{|\la-z|\le R} \frac{1}{z-\la}-\pi\mathfrak{c}_\La\bar a\bigg|^2=0
\]
for all $z,a\in \bC$.

Since $\La$ is stationary and
$\mathfrak{c}_\La=\mathfrak{c}_{T_a\La}$, 
for each $R\ge 1$ the random functions
\[
z \mapsto \sum_{|\la-z-a|\le R} \frac{1}{z+a-\la}
-\sum_{|\la-z|\le R} \frac{1}{z-\la}
+ \pi \mathfrak{c}_\La \bar a
\]
are stationary. But then the limiting random function 
$\zeta_\La(z+a)-\zeta_\La(z)$ 
is stationary as well (see 
Claim~\ref{claim:l_2_convergence_implies_convergence_in_dsitribution} below,
which we record separately for later use). 
In fact, the proof of the claim shows that
\begin{equation}
\label{eq:zeta-a-def1}
\zeta_\La(\cdot+a) - \zeta_\La(\cdot)=H_a(T_z\La)
\end{equation}
where $H_a$ is the $L^2(\Omega,\bP)$-limit 
\begin{equation}
\label{eq:zeta-a-def2}
H_a(\La)=\lim_{R\to\infty}\sum_{|\la|\le R}\Big(\frac{1}{a-\la}
+\frac{1}{\la}\Big).
\end{equation}
This completes the proof, modulo the proof of 
Claim~\ref{claim:l_2_convergence_implies_convergence_in_dsitribution}.
\end{proof}

\begin{claim}
\label{claim:l_2_convergence_implies_convergence_in_dsitribution}
For $R> 0$, let $f_{R,\La}$
be stationary random functions, and
assume that there exist a random function
$f_{\infty,\La}$ such that
\[
\lim_{R\to \infty } \bE\left[\left|f_{R,\La}(z) 
- f_{\infty,\La}(z)\right|^2\right]=0
\]
for all $z\in \bC$. Then $f_{\infty,\La}$ is stationary as well.
\end{claim}

\begin{proof}
By the definition of stationarity, there exist random variables $h_R$
such that, for $\bP$-a.e.\ $\La\in \Omega$ and for any $z\in\bC$, we have that
\[
f_{R,\La}(z)=h_R(T_z\La).
\] 
By assumption, there exists another
random variable, $h_\infty$, such that 
\[
\bE\left[|h_{R}(\La)-h_{\infty}(\La)|^2\right]
=\bE\left[|f_{R,\La}(0)-h_\infty(\La)|^2\right]\to 0,\qquad R\to\infty.
\]
Moreover, by the invariance of $\bP$ we have $(h_R-h_\infty)\circ T_z\to 0$ as well.
We claim that for a.e.\ $\La\in\Omega$ and for any $z\in\bC$,
\begin{equation}
\label{eq:equiv-rep-finfty}
f_{\infty,\La}(z)=h_\infty(T_z\La).
\end{equation}
Indeed, wherever $h_\infty$ is defined, 
\begin{align*}
f_{\infty,\La}(z)-h_\infty(T_z\La)&=f_{\infty,\La}(z)-f_{R,\La}(z)
+f_{R,\La}(z)-h_{R}(T_z\La) + h_{R}(T_z\La)-h_\infty(T_{z}\La)\\
&=f_{\infty,\La}(z)-f_{R,\La}(z)+h_{R}(T_z\La)-h_\infty(T_{z}\La).
\end{align*}
Since both $f_{\infty,\La}(z)-f_{R,\La}(z)$ and $(h_R-h_\infty)\circ T_z$ tend to
$0$ in $L^2(\Omega,\bP)$ and since $R>0$ was arbitrary, the right-hand side must vanish.
Consequently, the representation \eqref{eq:equiv-rep-finfty} for $f_{\infty,\La}$ follows.
\end{proof}

As a corollary to Theorem~\ref{thm1}, we observe that
the distribution of the random meromorphic function
\[
\wp_\La (z) \stackrel{\rm def}= - \zeta'_\La (z)
= \lim_{R\to\infty}\, \sum_{|\la|\le R} \frac1{(z-\la)^2}
\]
is stationary (as above, the convergence is in $L^2(\Omega, \bP)$
and is locally uniform in $z$). To obtain an equivariant 
representation of $\wp_\La$ similar to \eqref{eq:zeta-a-def1}
it suffices to note that
\begin{equation}
\label{eq:psi2-abs-lunar}
\limsup_{R\to\infty}\bE\Big|\sum_{|\la|\le R}\frac{1}{\la^2}
-\sum_{|\la-z|\le R}\frac{1}{\la^2}\Big|^2 \le 
\limsup_{R\to\infty}\frac{2}{(R-|z|)^4}\bE \big|n_\La(\bD(0,R)\setminus \bD(z,R)\big)\big|^2
=0,
\end{equation}
where the last equality follows from the crude bound
\begin{equation}
\label{eq:crude-bd-secondmoment}
\bE\big[n_\La\big(\bD(0,R)\setminus \bD(z,R)\big)\big]^2
\lesssim_zR^2,
\end{equation}
which in turn is obtained by a simple covering argument and the triangle inequality.
This yields the representation 
$\wp_\La(z)=P(T_z\La)$, where
\begin{equation}
\label{eq:stationary-rep-V}
P(\La)=\lim_{R\to\infty}\sum_{|\la|\le R}\frac{1}{\la^2}.
\end{equation}

Looking ahead a little, we note that the representation \eqref{eq27} of $\zeta_\La$
suggests that existence of the
$L^2(\Omega, \bP)$-limit
\begin{equation}\label{eq29}
\lim_{R\to\infty} \Psi_1(R) = \lim_{R\to\infty}\, \sum_{1\le |\la|\le R} \frac1{\la}
\end{equation}
becomes equivalent to existence of a stationary vector field
$V_\La (z)$. In its turn, it appears that existence of the
limit~\eqref{eq29} is easy to check on the spectral side.

\section{The invariant field}
\label{s:stat}
\subsection{Existence}
\begin{theorem}\label{thm2}
Let $\Lambda$ be a stationary point process in
$\bC$ having a finite second moment,
and denote by $\mathfrak c_\La$ the conditional
intensity of $\La$ on translation-invariant events.
Then the following are equivalent:

\smallskip\noindent {\rm (a)}
The spectral measure $\rho_\La$ satisfies
\[
\int_{\{0<|\xi|\le 1\}} \frac{{\rm d}\rho_\La (\xi)}{|\xi|^2} < \infty\,.
\]

\smallskip\noindent {\rm (b)}
The sum
\[
\Psi_1(R) = \sum_{1\leq |\la| \leq R} \frac{1}{\la}
\]
converges in $L^2(\Omega, \bP)$ to the limit $\Psi_1(\infty)$.

\smallskip\noindent {\rm (c)}
For some constant $\Psi\in L^2(\Omega,\bP)$, the random field
$\zeta_\La(z)-\Psi-\pi \mathfrak{c}_\La \bar{z}$
is stationary.
\end{theorem}

If any of the three conditions {\rm (a)--(c)} hold,
we choose the particular normalization
\begin{equation}
\label{eq:def-V}
V_\Lambda(z)=\zeta_\La(z)-\Psi_1(\infty)-\pi \mathfrak{c}_\La \bar{z}
\end{equation}
for the stationary random field.

We remark that in Condition {\rm (b)}, it is in fact sufficient 
to assume that $\Psi_1(R_j)$ is
convergent in $L^2(\Omega,\bP)$ along a sequence $R_j\to\infty$. 
Indeed, one can show that this condition directly implies {\rm (a)}.
We will not pursue the details here.

\begin{proof}
The implication {\rm (a)} $\Rightarrow$ {\rm (b)} is exactly Lemma~\ref{lem3}.
It remains to prove the implications {\rm (b)} $\Rightarrow$ {\rm (c)}
and {\rm (c)} $\Rightarrow$ {\rm (a)}.

\medskip

\noindent \underline{{\rm (b)} $\Rightarrow$ {\rm (c)}.}
We let $\Psi=\Psi_1(\infty)$ and proceed to show that $V_\Lambda$, as given in
\eqref{eq:def-V}, is stationary.
By the relation \eqref{eq:pointwise_L_2_limit_for_zeta} combined with
Condition (b) in the theorem, we get
\[
\lim_{R\to \infty} \bE\, \bigg|\, V_\La(z)
-\sum_{|\la - z|\leq R} \frac{1}{z-\la} \, \bigg|^2 = 0\, .
\]
Since $\La$ is stationary, the random functions
\[
z\mapsto \sum_{|\la - z|\leq R} \frac{1}{z-\la}
\]
are stationary for all $R\ge 1$.
The stationarity of $V_\La$ then follows from
Claim~\ref{claim:l_2_convergence_implies_convergence_in_dsitribution}. 

\medskip

\noindent \underline{{\rm (c)} $\Rightarrow$ {\rm (a)}.}
First we prove that, for any $\varphi\in \mathfrak D$, the random variable
\[
\zeta_{\La}(\varphi) \stackrel{{\rm def}}{=} \int_{\bC} \zeta_{\La} \varphi \, {\rm d}m
\]
is in $L^2(\Omega, \bP)$. Let $R$ be large enough so that
$\varphi$ is supported in $\frac{1}{2}R\bD$. Then
\begin{multline}\label{eq:split_in_zeta_phi}
\zeta_{\La}(\varphi)
= \sum_{|\la|\leq R} \mathcal C_{\varphi}(\la)
+\int_{\bC} \chi_{R}(\La,z) z^2 \,\varphi(z)\, {\rm d}m(z)
\\ \qquad + \Psi_1(R)\int_{\bC} \varphi(z)\, {\rm d}m(z)
+ \bigl(\Psi_2(R)-\Psi_2(\infty)\bigr)\int_{\bC}z\, \varphi(z)\,{\rm d}m(z),
\end{multline}
where $C_\varphi$ is the Cauchy transform of $\varphi$, and 
\[
\chi_R(\La,z) = \sum_{|\la|> R} \frac{1}{\la^2(z-\la)}.
\]
Since the Cauchy transform $\mathcal C_\varphi$ is a bounded on $\bC$,
$\Big|\sum_{|\la|\leq R} \mathcal C_{\varphi}(\la)\Big| \lesssim_{\varphi} n_{\La}(R\bD)$,
which implies that the first term on the right-hand side of \eqref{eq:split_in_zeta_phi}
belongs to $L^2(\Omega,\bP)$.
By Lemma~\ref{lem1}, we know that
\[
\sup_{z\in {\sf spt}(\varphi)} \bE\left|\chi_R(\La,z)\right|^2
\leq 2\sum_{|\la|\ge R} \frac{1}{|\la|^3}<\infty
\]
which implies that the second term in the sum is in $L^2(\Omega,\bP)$.
The random variable $\Psi_1(R)$ satisfies the bound
\[
|\Psi_1(R)|\le \sum_{1\le |\la|\le R}\frac{1}{|\la|}\le n_\La(\bD(0,R)),
\]
and the right-hand side has finite second moment by assumption,
so we conclude that the third term on the
right-hand side of \eqref{eq:split_in_zeta_phi}
also lies in $L^2(\Omega,\bP)$.
Finally, Lemma~\ref{lem2} tells us that the last
term on the right-hand side of \eqref{eq:split_in_zeta_phi}
is in $L^2(\Omega,\bP)$ as well,
and all together we get that $\zeta_{\La}(\varphi)\in L^2(\Omega,\bP)$.

Let now $\aV_\La(z)=\zeta_\La(z)-\Psi-\pi \mathfrak{c}_\La \bar{z}$.
Then $\aV_\La(\varphi)\in L^2(\bP)$
for any test function $\varphi$, which implies that $\aV_\La$ has a spectral measure
$\rho_{\aV_\La}$. In view of \cite[Theorem~5]{Yaglom}, the identity
$\bar\partial \aV_\La=\pi(n_\La-\mathfrak{c}_\La m)$ gives the relation
\begin{equation}
\label{eq:relation_between_spectral_V_and_lambda}
\rho_{\La} = |\xi|^2 \, \rho_{\aV_\La}\quad \text{on } \, \bC\setminus\{0\}.
\end{equation}
Since $\rho_{\aV_\La}$ is locally finite, the result then follows
by solving for $\rho_{\aV_\La}$ in \eqref{eq:relation_between_spectral_V_and_lambda}.
For the reader's convenience, we sketch a proof of
\eqref{eq:relation_between_spectral_V_and_lambda}.

\begin{claim}
\label{claim:Yaglom}
Assume that $\aV$ is a stationary random function such that
for some random constant $c\in L^2(\Omega,\bP)$
\[
\bar\partial\aV=\pi (n_\La-c\,m)
\]
in the sense of distributions, and suppose
moreover that $\aV(\varphi)\in L^2(\Omega,\bP)$ for any $\varphi\in \mathfrak{D}$.
Then
\begin{equation}
\rho_{\La} = |\xi|^2 \, \rho_{\aV_\La}\qquad \text{on } \, \bC\setminus\{0\}.
\end{equation}
\end{claim}

\begin{proof}[Proof of Claim~\ref{claim:Yaglom}]
Fix $\varphi \in \mathfrak D$.
Since $\bar \partial \aV = \pi (n_\La - c m)$, we have
\[
\pi (n_\La - c m)(\varphi) = - \aV(\bar \partial \varphi),
\]
whence,
\[
{\sf Var} \big[(n_\La - c m)(\varphi)\big]
= \frac{1}{\pi^2} {\sf Var} \big[\aV(\bar \partial \varphi)\big].
\]
Rewriting both sides in terms of the corresponding spectral measures
(note that the spectral measure of $n_\La - c m$ may differ from $\rho_\La$ by
at most an atom at the origin)
and using that
$\widehat{\bar \partial \varphi} (\xi) = \pi {\rm i} \xi \, \widehat{\varphi}(\xi)$,
we get
\begin{equation}
\label{eq:relation_between_spectral_V_and_lambda_integral_form}
\int_{\bC\setminus \{0\}} |\widehat{\varphi}|^2 \,
{\rm d}\rho_{\La}+a|\widehat{\varphi}(0)|^2
= \int_{\bC} |\xi|^2|\widehat{\varphi}|^2 \, {\rm d}\rho_{\mathcal{V}_{\La}},
\end{equation}
for some constant $a$.
Next, we recall that the Fourier transforms of functions in $\mathfrak D$ are dense
in the Schwartz space $\mathcal{S}$ (see Remark~\ref{rem:L2-extension}).
For an arbitrary compact set $K\subset \bC\setminus\{0\}$,
we approximate its indicator function $\done_K$ by a
uniformly bounded sequence $(\varphi_n) \subset \mathcal{S}$,
converging to $\done_K$ pointwise. Passing to the limit in
\eqref{eq:relation_between_spectral_V_and_lambda_integral_form}, we get
\[
\rho_{\La}(K) = \int_{K} |\xi|^2\, {\rm d} \rho_{\aV_{\La}}(\xi),
\]
which gives \eqref{eq:relation_between_spectral_V_and_lambda}.
\end{proof}

With the proof of Claim~\ref{claim:Yaglom} complete, we are done with the
proof of Theorem~\ref{thm2}.
\end{proof}

Let us note
that the possible atom at the origin of the spectral measure
$\rho_\La$ is irrelevant for the spectral condition~(a).

\medskip
The spectral measures of the zero process of GEFs, of the limiting Ginibre ensemble and
of the stationarized random perturbation of the lattice satisfy spectral condition~(a).
Hence, for these point processes, the generalized random function $V_\La$ is stationary,
while for the Poisson process it only has  stationary increments.
All this can be proved directly for each of these processes.
Theorem~\ref{thm2} provides us with a unified reason for this phenomenon.

\begin{remark}\label{rem-cond-kappa}
The following set of conditions on the reduced covariance measure $\kappa_\La$ yields
the spectral condition~(a) in Theorem~\ref{thm2}:
\begin{itemize}
\item[($\kappa_1$)] existence of the $1$st moment: $\displaystyle \int_\bC |s|\,
{\rm d}|\kappa_\La|(s) < \infty$;
\item[($\kappa_2$)] the zeroth sum-rule: $\kappa_\La(\bC)=0$.
\end{itemize}
\end{remark}

Indeed, existence of the first moment of $\kappa_\La$ yields that the spectral measure $\rho_\La$
is absolutely continuous with a non-negative $C^1$-smooth density $h$.
By condition~($\kappa_2$), $h$ vanishes at the origin. Since $h$ is
continuously differentiable, we conclude that $h(\xi) = O(|\xi|)$ as $\xi\to 0$, which yields
the spectral condition~(a).

The zeroth sum-rule ($\kappa_2$) is known to imply suppressed fluctuations of
$n_\La$ (see, for instance,  \cite[Section~1C]{Martin}). Ghosh and Lebowitz 
\cite{GhoshLebowitzRigidity} observed that the combination of $(\kappa_2)$ 
with a stronger than $(\kappa_1)$ decay of correlations
yields an interesting geometric property of $\La$ known as number-rigidity.

\subsection{Uniqueness}
\begin{theorem}\label{thm3}
Let $\aV_\La$ be a generalized random function satisfying the following properties:

\smallskip\noindent {\rm ($\alpha$)} It is stationary.

\smallskip\noindent {\rm ($\beta$)} There exists a random constant $c$ such that
$\aV_\La (z) +\pi c\bar{z}$ is meromorphic with with poles exactly at $\Lambda$,
all simple and with unit residue.

\smallskip\noindent {\rm ($\gamma$)} For any test function $\varphi\in \mathfrak{D}$,
the random variable
\[
\aV_\La (\varphi) = \int_\bC \aV_\La \varphi\, {\rm d}m
\]
lies in $L^2(\Omega,\bP)$.

\smallskip
\noindent Then the spectral condition {\rm (a)} of Theorem~\ref{thm2} holds, and the
random fields $\mathcal{V}_\La$ and $V_\La$
differ by a constant in $L^2(\Omega,\bP)$ which is measurable with respect 
to $\mathcal{F}_{\text{\sf inv}}$, the sigma-algebra of translation invariant events.
\end{theorem}

\begin{proof}
By Claim~\ref{claim:Yaglom}, the spectral measure $\rho_{\aV}$ of $\aV_\La$
agrees with $|\xi|^{-2}\rho_\La$ outside the origin, and hence the spectral
condition {\rm (a)} in Theorem~\ref{thm2} holds.
We may therefore speak about the random function $V_\La$, and we have
\[
\done_{\bC\setminus\{0\}}(\xi){\rm d}\rho_{V_\La}(\xi)
=\done_{\bC\setminus\{0\}}(\xi){\rm d}\rho_{\aV_\La}(\xi).
\]
We next observe that
\begin{align*}
\bE\bigg|\int_{\{|z|=R\}} \aV_\La (z)\, {\rm d}z\bigg|^2
&=\int_{\bC}|\widehat\nu_R|^2 {\rm d}\rho_{\aV_\La}(\xi)\\
&=\int_{\bC}|\widehat\nu_R|^2 {\rm d}\rho_{V_\La}(\xi) +
\big(\rho_{\aV_\La}(\{0\})-\rho_{V_\La}(\{0\})\big)|\widehat{\nu}_R(0)|^2\, ,
\end{align*}
where $\nu_R$ is the current of integration $\nu_R(f):=\int_{\{|z|=R\}}f(z){\rm d}z$
with respect to the differential ${\rm d}z$
along $|z|=R$. Hence
\[
\widehat{\nu}_R(\xi)=\int_{\{|z|=R\}}e^{-2\pi i \xi\cdot z}{\rm d}z,
\]
so it follows that $\widehat{\nu}_R(0)=0$. As a consequence, the atom at the
origin does not matter, so by
repeating the same calculation backwards we arrive at
\[
\bE\, \bigg|\int_{\{|z|=R\}} \aV_\La (z)\, {\rm d}z\bigg|^2
=\bE \, \bigg|\int_{\{|z|=R\}} V_\La (z)\, {\rm d}z\bigg|^2.
\]
Because $\int_{|z|=R} \bar z \, {\rm d}z = 2\pi {\rm i} R^2$,
the residue theorem gives us that
\begin{equation}
\label{eq:form-integral-curve}
\frac{1}{2\pi {\rm i}} \int_{\{|z|=R\}}\aV_\La(z)\,{\rm d}z=\frac{1}{2\pi {\rm i}}
\int_{\{|z|=R\}}(\aV_\La (z)+\pi c \bar z)\,{\rm d}z-\pi c R^2=n_{\La}(R\bD)-\pi cR^2
\end{equation}
and the same holds with $\aV_\La$ replaced by $V_\La$ 
and $c$ replaced by $\pi\mathfrak{c}_\La$. But then
\[
\bE\Big[\Big|\frac{n_\La(R\bD)}{R^2} - \pi c \, \Big|^2\Big]=
\bE\Big[\Big|\frac{n_\La(R\bD)}{R^2} - \pi \mathfrak{c}_\La\Big|^2\Big],
\]
and juxtaposing this identity with the fact that
\begin{equation}
\label{eq:as-formula-cond-int}
\lim_{R\to \infty} \bE\Big[\Big|\frac{n_\La(R\bD)}{R^2}
-\pi\mathfrak{c}_\La\Big|^2\Big] = 0,
\end{equation}
we get that $c= \mathfrak{c}_\La$, almost surely.

Let $G_\La = V_\La - \aV_\La$, so that
$G_\La$ is a stationary random entire function.
If we choose the test-function $\varphi$ to be radial with total integral 1, then
\begin{align*}
\left(V_\La-\aV_\La\right)(T_z\varphi)&= \int_\bC \varphi(z+w)G_\La(w) \, {\rm d}m(w)\\
&=\int_{0}^\infty\varphi(r)r\int_{0}^{2\pi}G_\La(z+r e^{i\theta}) \, {\rm d}\theta \, {\rm d}r
\\
&=2\pi\, G_\La(z)\int_{0}^\infty\varphi(r)r\,{\rm d}r=2\pi\, G_\La(z),
\end{align*}
where we used the mean value property of holomorphic functions to
arrive at the second equality. The LHS is the 
difference of two random variables in $L^2(\Omega,\bP)$,
and the variance of each term is independent of $z$.
Hence $G_\La$ is a random entire function with
\begin{equation}
\label{eq:first-moment-F}
\sup_{z\in \bC} \bE\,|G_\La(z)|^2 < \infty.
\end{equation}
Armed with this, we get the bound
\[
\bE\,\bigg|\int_{\bC} \frac{|G_\La(z)|}{1+|z|^{5/2}}\, {\rm d}m(z)\bigg|
\lesssim \int_{1}^{\infty} \frac{{\rm d}t}{t^{3/2}} <\infty,
\]
which, together with positivity, implies that the random variable
\[
\int_{\bC} \frac{|G_\La(z)|}{1+|z|^{5/2}}\, {\rm d}m(z)
\]
is finite almost surely. The mean value property implies that
\begin{align*}
|G_\La(\zeta)| \leq \frac{1}{\pi |\zeta|^2}
\int_{\{|z-\zeta|\leq |\zeta|\}}& |G_\La(z)| \, {\rm d}m(z) \\
&\lesssim \frac{|\zeta|^{5/2}}{|\zeta|^2}
\int_{\bC} \frac{|G_\La(z)|}{1+|z|^{5/2}} {\rm d} m(z)
\lesssim |\zeta|^{1/2}
\end{align*}
for all $|\zeta|\ge 1$, so in view of Liouville's 
theorem $G_\La$ is almost surely constant.

Finally, to see that $G_\La(0)$ is measurable with 
respect to $\mathcal{F}_{\sf{inv}}$, we note that
\[
G_\La(0) = \frac{1}{\pi R^2} \int_{\{|z|\le R\}} 
\Big(V_\La(z) - \mathcal{V}_\La(z)\Big) \, {\rm d}m(z)
\]
for all $R\ge 1$. By Wiener's ergodic theorem 
\cite[Theorem~3]{Becker}, we get that
\[
\lim_{R\to \infty} \bE \, \bigg| \frac{1}{\pi R^2} 
\int_{\{|z|\le R\}} \Big(\mathcal{V}_\La(z)
- V_\La(z)\Big) \, {\rm d}m(z) - \bE\big[\mathcal{V}_\La(0) - V_\La(0) 
\mid \mathcal{F}_{\sf{inv}}\big]\bigg| = 0\, ,
\]
which immediately implies that
$G_\La(0) = \bE\big[G_\La(0) \mid \mathcal{F}_{\sf{inv}}\big]$ in $L^1(\Omega,\bP)$. 
That is, $G_\La(0)$ is measurable with respect 
to $\mathcal{F}_{\sf{inv}}$, and we are done.
\end{proof}

\begin{remark}
We will note that, following~\cite{BGLS}, one can significantly
relax the condition ($\gamma$)
in Theorem~\ref{thm3} (cf.\ \cite[Theorem 3A]{BGLS}).
We will not pursue this here.

Note also that the above proof shows that the assumption $(\alpha)$ in 
Theorem~\ref{thm3} is stronger than necessary. 
What is really needed is that $V_\La-\mathcal{V}_\La$
has some uniformly bounded moment.
\end{remark}

\begin{remark}\label{remark:ergodic_implies_independence}
If we assume that the point process $\La$ is ergodic, i.e., 
that $\mathcal{F}_{{\sf inv}}$ is trivial,
then any random function
$\mathcal{V}_\La$ which satisfies conditions ($\alpha$), ($\beta$) and ($\gamma$) 
differs from $V_\La$ by a deterministic constant. 
Indeed, the function $G_\La = V_\La - \mathcal{V}_\La$ was shown to be a constant in the above proof, 
and measurable with respect to $\mathcal{F}_{{\sf inv}}$.
\end{remark}

If the field $V_\La$ satisfies only conditions ($\alpha$) and ($\beta$) of Theorem~\ref{thm3}, then it
is defined up to a random entire function with translation-invariant distribution.
As was discovered by Weiss~\cite{Weiss} such entire functions do exist.
Developing his idea, one can show that, somewhat paradoxically,
{\em for any stationary process $\La$, there exists a random field
$V_\La$ satisfying conditions} ($\alpha$) {\em and} ($\beta$).
It is worth mentioning that these ``exotic'' random fields behave quite
wildly (cf. Buhovsky-Gl\"ucksam-Logunov-Sodin~\cite{BGLS}), as opposed to the ``tame'' ones
from Theorem~\ref{thm2}.

\subsection{Fluctuations}
The relations
$\partial_{\bar z} V_\La = \pi (n_\La - \mathfrak c_\La)$ and
$\partial_{z} V_\La = -\wp_\La$
allow one to readily relate the spectral 
measures and the reduced covariance measures
of these functions to the ones of the point process $\La$.

\subsubsection{The functions \texorpdfstring{$\mathsf\Delta_a \zeta_\La$}{DeltaZeta}
and \texorpdfstring{$\wp_\La$}{pe}}
Here we only assume that $\La$ is a stationary random
planar point process having finite second moment.
Then, by Theorem~\ref{thm1}, the random meromorphic functions
$\mathsf\Delta_a \zeta_\La (z) =
\zeta_\La (z+a) - \zeta_\La (z)$, $a\in\bC $,
and
$\wp_\La (z) = \lim_{a\to 0} \frac1a\, \mathsf\Delta_a \zeta_\La (z)$
are stationary. Since their second moments are infinite pointwise, we treat them
as generalized stationary random processes on the space $\mathfrak{D}$
of test-function by
$ \mathsf\Delta_a \zeta_\La (\varphi) = \zeta_\La (T_{-a}\varphi - \varphi) $,
where $T_w\varphi (z) = \varphi (z+w)$, and by
$ \wp_\La (\varphi) = - (\partial_z \zeta_\La)(\varphi)
= \zeta_\La (\partial_z \varphi)$.

\begin{theorem}\label{thmA}
Let $\La$ be a stationary point process in $\bC$ having a finite second moment.
Then,
\[
{\rm d}\rho_{\mathsf\Delta_a\zeta_\La}(\xi)
= \done_{\bC\setminus\{0\}}(\xi)\, \frac{|1-e^{2\pi{\rm i} a\cdot\xi}|^2}{|\xi|^2}\,
{\rm d}\rho_\La (\xi) + \pi^2|a|^2\rho_\La(\{0\})\delta_0(\xi)\,,
\]
and as a consequence
${\rm d}\rho_{\mathsf\Delta_a\zeta_\La-\pi \mathfrak{c}_\La\bar{a}}(\xi)
=\done_{\bC\setminus\{0\}}(\xi)\, \frac{|1-e^{2\pi{\rm i} a\cdot\xi}|^2}{|\xi|^2}\,
{\rm d}\rho_\La (\xi)$.
Moreover, we have
\[
{\rm d}\rho_{\wp_\La}(\xi) = 
\pi^2 \done_{\bC\setminus\{0\}}(\xi) {\rm d}\rho_\La (\xi)\,.
\]
\end{theorem}

\begin{proof}
We begin by determining the spectral measure of ${\sf \Delta}_a\zeta_\La$.
We have $\bar\partial {\sf \Delta}_a \zeta_\La=\pi(n_{T_a \La}-n_{\La})$,
and the spectral measure for $\pi(n_{T_a \La}-n_{\La})$ equals
$\pi^2|1-e^{2\pi {\rm i}\xi\cdot a}|^2{\rm d}\rho_\La(\xi)$.
For any $\varphi\in\mathfrak D$, we have
\begin{equation*}
{\Delta}_a\zeta_{\La}(\bar \partial\varphi) = -\pi (n_{T_a\La}-n_\La)(\varphi),
\end{equation*}
which in turn implies that ${\sf Var}\big({\Delta}_a\zeta_{\La}(\bar \partial\varphi)\big)
= \pi^2\, {\sf Var}\big((n_{T_a \La}-n_\La) (\varphi)\big)$.
Moving to the Fourier side, we get that for all $\varphi\in \mathfrak{D}$,
\begin{equation*}
\int_{\bC} |\xi|^2 |\widehat{\varphi}|^2 \, {\rm d}\rho_{{\sf \Delta}_a\zeta_\La}
= \int_{\bC} \big|1-e^{2\pi {\rm i} a\cdot \xi}\big|^2
|\widehat{\varphi}|^2 \, {\rm d}\rho_{\La},
\end{equation*}
and hence,
\begin{equation}
\label{eq:Deltaa}
\rho_{{\sf \Delta}_a\zeta_\La} = 
\frac{\big|1-e^{2\pi {\rm i} a\cdot \xi}\big|^2}{|\xi|^2} \, \rho_{\La}\,
\qquad \text{on} \ \, \bC\setminus\{0\}.
\end{equation}
It remains to analyze $\rho_{{\sf \Delta}_a\zeta_\La}(\{0\})$.
This will involve a computation which we defer to Appendix~\ref{s:atom1}. 
In particular, these computations will reveal that the atom is
only present in the rather exotic case when $\La$ \say{hyperfluctuates}, i.e.,
when $\rho_\La$ has an atom at the origin to begin with.
With this, we conclude the proof of the first part.

Turning to the spectral measure of $\wp_\La$,
note that for any $\varphi\in \mathfrak D$, we have
\begin{equation*}
\wp_{\La}(\bar\partial\varphi) = -\partial\zeta_{\La}(\bar\partial\varphi)
= -\bar\partial \zeta_{\La} ( \partial \varphi) = \pi n_\La (\partial \varphi)\, ,
\end{equation*}
which in turn implies that ${\sf Var}[\wp_{\La}(\bar\partial\varphi)]
= \pi^2\, {\sf Var}[n_\La (\partial \varphi)]$.
Moving to the Fourier side, we get that, for all $\varphi\in \mathfrak{D}$,
\begin{equation*}
\int_{\bC} |\xi|^2 |\widehat{\varphi}|^2 \, {\rm d}\rho_{\wp_\La}
= \pi^2 \int_{\bC} |\xi|^2 |\widehat{\varphi}|^2 \, {\rm d}\rho_{\La}
\end{equation*}
which gives that
\begin{equation}
\label{eq:spectral_measure_for_pe_without_origin}
\rho_{\wp_{\La}} = \pi^2 \rho_{\La}\, \qquad \text{on} \ \, \bC\setminus\{0\}\, .
\end{equation}
To conclude the proof, it only remains to show that $\rho_{\wp_\La}$ has no
mass at the origin. We defer this computation to Appendix~\ref{s:atom2}.
\end{proof}

\subsubsection{The vector field \texorpdfstring{$V_\La$}{V}}

\begin{theorem}\label{thmB}
Suppose $\Lambda$ is a stationary point process in $\bC$
satisfying any of the equivalent conditions
in Theorem~\ref{thm2}. Then,
\[
{\rm d}\rho_{V_\La}(\xi) = \done_{\bC\setminus \{0\}}(\xi)\,
\frac{{\rm d}\rho_\La (\xi)}{|\xi|^2}\,.
\]
\end{theorem}

\begin{proof}
Since $\bar\partial V_\La=\pi(n_\La-\mathfrak{c}_\La m)$,
it follows from Claim~\ref{claim:Yaglom} that
\[
{\rm d}\rho_{V_\La}(\xi)=\frac{{\rm d}\rho_\La(\xi)}{|\xi|^2},\qquad \text{on } 
\; \bC \setminus\{0\}.
\]
It only remains to check that $\rho_{V_\La}(\{0\})=0$, which we again
postpone to Appendix~\ref{s:atom3}.
\end{proof}

Theorem~\ref{thmB} yields a useful representation of the reduced
covariance measure $\kappa_{V_\La}$ of the field $V_\La$.
We denote by $ U^\mu$ the logarithmic potential of a signed measure $\mu$,
\[
U^\mu(z) = \int_\bC \log\frac1{|z-s|}\, {\rm d}\mu (s)
\]
(provided that the integral on the RHS exists).

\begin{proposition}
Suppose that reduced covariance measure $\kappa_\La$ satisfies assumptions
$(\kappa_1)$ and $(\kappa_2)$ of Remark~\ref{rem-cond-kappa}.
Then the reduced covariance measure of $V_\La$ equals
\[
{\rm d}\kappa_{V_\La} = 2\pi U^{\kappa_\La}\, {\rm d}m\,.
\]
\end{proposition}

\begin{proof}
To prove the proposition, we show that
$2\pi \widehat{U^{\kappa_\La}} = \done_{\{\xi\ne 0\}} |\xi|^{-2} \widehat{\kappa_\La}$.
We treat both sides as Schwartz distributions and understand the
Fourier transforms in the sense of distributions.

First, we note that $\Delta U^{\kappa_\La} = -(2\pi)^{-1} \kappa_\La$, and therefore,
$\widehat{\kappa_\La} = 2\pi |\xi|^2 \widehat{U^{\kappa_\La}}$.
Hence, on the test functions $\varphi$ with $0\notin{\rm spt}(\varphi)$,
the distributions $2\pi \widehat{U^{\kappa_\La}}$ 
and $|\xi|^{-2} \widehat{\kappa_\La}$ coincide.
I.e., the distribution 
$\nu=2\pi \widehat{U^{\kappa_\La}} 
- |\xi|^{-2} \widehat{\kappa_\La}$
is supported by the origin, and therefore, 
is a (finite) linear combination
of the delta-function and its partial 
derivatives. We need to show that $\nu=0$.

Fix a non-negative $C^\infty$-smooth function
$\gamma$ with a compact support, normalized by
$\displaystyle\int_{\bC} \gamma\, {\rm d}m =1$,
and let $\gamma_a(z)=a^{-2}\gamma (z/a)$, $a>0$.
This is a convolutor on the Schwartz space $\mathcal{S}'$
of tempered distributions, and, for any $f\in \mathcal{S}'$,
$f \ast \gamma_a \to f$ in $\mathcal{S}'$, as $a\to 0$.
The Fourier transform $\widehat{\gamma_a}=\widehat{\gamma}(az)$
is a $C^\infty$-smooth, fast decaying multiplier on the
Fourier side of $\mathcal{S}'$, boundedly tending to $1$
pointwise, as $a\to 0$.
Consider the product
\[
2\pi \widehat{U^{\kappa_\La}} \cdot \widehat{\gamma_a}
=|\xi|^{-2} h_\La \cdot \widehat{\gamma_a} + \nu\cdot \widehat{\gamma_a},
\]
where $h_\La$ is the density of $\widehat{\kappa_\La}$.
The inverse Fourier transform of the LHS equals
$2\pi U^{\kappa_\La}\ast\gamma_a = U^{\kappa_\La\ast\gamma_a}$.
The measure $\kappa_\La\ast\gamma_a$ enjoys
the same properties $(\kappa_1)$ and $(\kappa_2)$ as $\kappa_\La$.
Hence, the logarithmic potential
$ U^{\kappa_\La\ast\gamma_a} $ tends to $0$ as $z\to\infty$.
By the Riemann-Lebesgue lemma, the inverse Fourier transform of
$|\xi|^{-2} h_\La \cdot \widehat{\gamma_a} $ also tends to zero as $z\to\infty$.
Hence, the same holds for the inverse Fourier transform of
the distribution $\nu\cdot\widehat{\gamma_a}$.
But the inverse Fourier transform of the distribution $\nu$
is a polynomial (of ${\rm Re}(z)$ and ${\rm Im}(z)$),
and therefore, the inverse Fourier transform of the
distribution $\nu\cdot\widehat{\gamma_a}$ is also a polynomial.
We conclude that $\nu\cdot\widehat{\gamma_a} = 0$ for all $a>0$,
and therefore, $\nu=0$.
\end{proof}

\begin{remark}
Recalling that $\kappa_\La = \tau_\La + c_\La \delta_0$, we see that
\[
U^{\kappa_\La}(z) = c_\La \log\frac1{|z|} + U^{\tau_\La}(z)\,,
\]
that is, under the assumptions of Remark~\ref{rem-cond-kappa}, the covariance kernel of
$V_\La$ always blows up logarithmically at the origin.
\end{remark}

\begin{remark}\label{rem-isotropic}
In the case when the distribution of the point process $\La$ is also rotationally invariant
the density of  the reduced covariance measure $\kappa_{V_\La}$ has a simpler expression.
We assume that
${\rm d}\tau_\La (s) = k(t)t\,{\rm d}t\, {\rm d}\theta$, $s=te^{{\rm i}\theta}$.
Then conditions
$(\kappa_1)$ and $(\kappa_2)$ of Remark~\ref{rem-cond-kappa} can be re-written as
\begin{itemize}
\item[($\tau_1$)] $\displaystyle \int_0^\infty t^2|k(t)|\, {\rm d}t < \infty$;
\item[($\tau_2$)] $\displaystyle \int_0^\infty tk(t)\, {\rm d}t = - (2\pi)^{-1}c_\La$.
\end{itemize}
In this case,
\begin{align*}
U^{\kappa_{\La}}(z) &= c_\La \log\frac1{|z|}
+ \int_0^\infty \Bigl( \log\frac1{|z-te^{{\rm i}\theta}|} \, {\rm d}\theta \Bigr)
k(t) t\, {\rm d}t \\
&= c_\La \log\frac1{|z|} + \int_0^\infty
2\pi \bigl(\, \log\frac1{|z|} -  \log_+\frac{t}{|z|}\, \bigr) k(t) t\, {\rm d}t \\
&\stackrel{(\tau_2)}= - 2\pi \int_{|z|}^\infty \log\frac{t}{|z|} k(t)t\, {\rm d}t\,.
\end{align*}
Thus, the density of $\kappa_{V_\La}$ equals
\begin{equation}
\label{eq:dens-kappa}
- 4\pi^2 \int_{|z|}^\infty \log\frac{t}{|z|} k(t)t\, {\rm d}t\,.
\end{equation}
\end{remark}

\section{The random potential}
\label{s:potential}
\subsection{The potential \texorpdfstring{$\Pi_\La$}{Pi}}
Assuming that the equivalent conditions of Theorem~\ref{thm2} hold,
we will define a random potential $\Pi_\La$ such that
$\partial_z \Pi_\La = \tfrac12\, V_\La $, and therefore,
$\Delta \Pi_\La = 2\pi (n_\La - \mathfrak c_\La\,m)$
(both relations are understood in the sense
of distributions). Since the field $V_\La$ is stationary,
this will yield that the potential $\Pi_\La$ has stationary increments.
The existence of the potential
$\Pi_\Lambda$ with this property, in turn, shows that the vector field $V_\La$ is
stationary (and therefore, yields conditions (a) and (b) in Theorem~\ref{thm2}).

Note that it is possible to prove an analogous result to Theorem~\ref{thm2},
which states that the existence of a \emph{stationary} potential
is equivalent to the stronger spectral condition
\[
\int_{|\xi|>0}\frac{{\rm d}\rho_\La(\xi)}{|\xi|^4}<\infty
\] 
(cf. Theorem~\ref{thmC} below).
We will not pursue the details here.

\medskip
We start with the entire function represented by the Hadamard product
\[
F_\La (z) = \exp\bigl[ -\Psi_1(\infty)z - \frac12\, \Psi_2(\infty)z^2 \bigr]\,
\cdot \prod_{|\la|<1} (\la-z)
\prod_{|\la|\ge 1} \left(\frac{\la-z}\la\,
\exp\Bigl[\,\frac{z}\la\, + \frac{z^2}{2\la^2}\, \Bigr]\right)\,,
\]
and note that, for each $R>1$,
\begin{multline*}
F_\La (z) = \prod_{|\la|<1} (\la-z) \prod_{1\le |\la|\le R} \frac{\la-z}\la \\
\times \exp\Bigl[ (\Psi_1(R)-\Psi_1(\infty)) z
+ \frac12\, (\Psi_2(R)-\Psi_2(\infty)) z^2 \Bigr]
\prod_{|\la|>R} e^{H(z/\la)}\,,
\end{multline*}
where $H(w) = - \sum_{k\ge 3} w^k/k$. As $R\to\infty$,
the third and fourth factors on the RHS tend to $1$ in
$L^2(\Omega, \bP)$ and locally uniformly in $z$, and therefore,
\begin{equation}
\label{eq:F-La-rep}
F_\La (z) = \prod_{|\la|<1} (\la-z) \lim_{R\to\infty}\,
\prod_{1\le |\la|\le R} \frac{\la-z}\la\,.
\end{equation}
We define $\Pi_\La (z)
\stackrel{\rm def}= \log|F_\La (z)| - \tfrac12\, \pi \mathfrak c_\La |z|^2$.
Then, by a straightforward inspection, we get that
\begin{equation}\label{eq:partialPi}
\partial_z \Pi_\La
= \frac12\, \bigl( \zeta_\La - \Psi_1(\infty)
- \pi\mathfrak c_\La \bar z \bigr)
= \tfrac12\, V_\La.
\end{equation}

\begin{remark}\label{rem-argument}
Under the assumptions of Theorem~\ref{thm2}, the quotient
\[
\Bigl|\,
\frac{F_\La(z+a)}{F_\La(z)}\, \Bigr| \, \exp\bigl[
-\frac12\, \pi \mathfrak c_\La (z\bar a + a\bar z + |a|^2) \bigr], \quad a\in\bC,
\]
has a stationary distribution (as a function of $z$).
An interesting characteristic of the point process $\Lambda$ is the distribution
of the phase
\[
\arg F_\La(z+a)/F_\La(z) = \arg F_\La(z+a) - \arg F_\La(z)\,.
\]
To properly define
this quantity, we fix a curve $\Gamma$ connecting the points
$z$ and $z+a$, and consider the increment of 
the argument of $F_\La$ along $\Gamma$
\[
\frac1{2\pi} \Delta_\Gamma \arg F_\La = {\rm Im}\, \frac1{2\pi}\,
\int_\Gamma (\zeta_\La (z) -\Psi_1(\infty))\, {\rm d}z\,.
\]
This quantity was considered by Buckley-Sodin in
\cite{BuckleySodin} when $F$ is replaced by the GEF.
Equivalently, one can consider the flux of the gradient
field of the potential $\Pi_\La$ through the curve $\Gamma$.

In \cite{SWY-II}, we study the asymptotic variance of this
quantity under dilations of $\Gamma$.
In the special case when $\Gamma$ is a Jordan curve
we recall that the change in argument coincides
with the charge fluctuation around the mean in the domain
enclosed by $\Gamma$.
\end{remark}

\subsection{The covariance structure of \texorpdfstring{$\Pi_\La$}{Pi}}
Let $\mathsf \Delta_a \Pi_\La (z) \stackrel{\rm def}
= \Pi_\La (z+a) - \Pi_\La (z)$, for $a\in\bC$.

\begin{theorem}\label{thmC}
Suppose $\Lambda$ is a stationary point process in 
$\bC$ satisfying any of equivalent conditions
in Theorem~\ref{thm2}. Then, ${\sf \Delta}_a\Pi_\La$ is stationary and
\[
\rho_{\mathsf \Delta_a \Pi_\La}(\xi) = \frac{1}{4}\, \done_{\bC\setminus \{0\}}(\xi)\,
\frac{|1-e^{2\pi {\rm i}a\cdot\xi}|^2}{|\xi|^4}\, \rho_\La (\xi)
+\frac{\pi^2|a|^4}{4}\rho_\La(\{0\})\,\delta_0(z).
\]
\end{theorem}

\begin{proof}
We first claim that ${\sf \Delta}_a\Pi_\La(z)$
can be written as 
\begin{equation}
\label{eq:stat-Pi-a1}
{\sf \Delta}_a\Pi_\La(z)=Q_a(T_z\La)-\frac12\pi\mathfrak{c}_\La|a|^2,
\end{equation}
where $Q_a$ is the $L^2(\Omega,\bP)$-limit
\begin{equation}
\label{eq:stat-Pi-a2}
Q_a(\La)=\lim_{R\to\infty}\sum_{|\la|\le R}\Big(\log|z-a|-\log|z|\Big),
\end{equation}
which in particular says that ${\sf \Delta}_a\Pi_\La$ is stationary. 
To verify \eqref{eq:stat-Pi-a1}--\eqref{eq:stat-Pi-a2},
we start with the representation \eqref{eq:F-La-rep}, which in view of the spectral condition
gives that
\[
{\sf \Delta}_a\Pi_\La(z)=\lim_{R\to\infty}
\sum_{|\la|\le R}\big(\log|z-(\la-a)|-\log|z-\la|\big)
-\frac12\pi\mathfrak{c}_\La\big(2\Re (\bar{a}z)+|a|^2\big),
\]
where the limit is taken in $L^2(\Omega,\bP)$. 
We readily rewrite this as
\[
{\sf \Delta}_a\Pi_\La(z)=\lim_{R\to\infty}
\Big[\sum_{|\la-z|\le R}\big(\log|z-\la+a)|-\log|z-\la|\big)
+e_R(\La,z)\Big]-\frac{1}{2}\pi \mathfrak{c}_\La |a|^2,
\]
where $e_R(\La,z)$ is an \say{error term} given by
\[
e_R(\La,z)=\sum_{|\la|\le R}\big(\log|z-\la+a|-\log|z-\la|\big) 
\,- \hspace{-5pt}\sum_{|\la-z|\le R}\big(\log|z-\la+a|-\log|z-\la|\big) 
- \pi\mathfrak{c}_\La\Re(\bar{a}z).
\]
By Claim~\ref{claim:l_2_convergence_implies_convergence_in_dsitribution}, 
it only remains to prove that $\bE\big[|e_R(\La,z)|^2\big]\to 0$ as $R\to\infty$.
To see this, we will relate $e_R(\La,z)$ to the sums appearing in the
Lunar lemma (Lemma~\ref{lemma:lunar_domains}).
Note first that 
\begin{equation}
\label{eq:log-expansion}
\log|z-\la+a|-\log|z-\la|=\Re \frac{a}{z-\la}
+O\Big(\frac{1}{|z-\la|^{2}}\Big)
\end{equation}
as $|z-\la|\to\infty$.
Using this expansion, we rewrite $e_R(\La,z)$ as
\begin{align}
\label{eq:eRLaz}
e_R(\La,z)&=\Re \Big[a\Big(\sum_{|\la|\le R}\frac{1}{z-\la}
-\sum_{|\la-z|\le R}\frac{1}{z-\la}\Big)\Big]-\pi \mathfrak{c}_\La \Re(a\bar{z})
+\widetilde{e}_R(\La,z)\nonumber \\
&=\Re \Big[a\Big(\sum_{|\la|\le R}\frac{1}{z-\la}
-\sum_{|\la-z|\le R}\frac{1}{z-\la}-\pi\mathfrak{c}_\La\bar{z}\Big)\Big]+\widetilde{e}_R(\La,z)
\end{align}
where the new error term $\widetilde{e}_R(\La,z)$ is
\begin{align}
\label{eq:e-tilde}
\widetilde{e}_R(\La,z)&=\sum_{|\la|\le R}l_a(z-\la)
-\sum_{|\la-z|\le R}l_a(z-\la)\\
&=\sum_{\la\in \bD(0,R)\setminus \bD(z,R)}l_a(z-\la)
-\sum_{\la\in \bD(z,R)\setminus \bD(0,R)}l_a(z-\la),
\end{align}
and where
\[
l_a(w)=\log|w+a|-\log|w|-\Re\Big(\frac{a}{w}\Big)=O(|w|^{-2}).
\]
The first term on the RHS of \eqref{eq:eRLaz} tends to
zero in $L^2(\Omega,\bP)$ by Lemma~\ref{lemma:lunar_domains}.
Moreover, for any point $\la$ in the symmetric difference 
\[
S_R\overset{\sf def}{=}\big(\bD(0,R)\setminus\bD(z,R)\big)\cup \big(\bD(z,R)\setminus\bD(0,R)\big),
\] 
we have $|l_a(z-\la)|=O(R^{-2})$, so by the upper bound 
\eqref{eq:crude-bd-secondmoment} of $\bE[n_\La(S_R)]^2$, 
$\widetilde{e}_R(\La,z)\to 0$ in $L^2(\Omega,\bP)$ as $R\to\infty$.
This completes the proof.

Turning to the spectral measure, 
note that by \eqref{eq:partialPi}, we have
\[
{\sf \Delta}_a \Pi_\La(\partial\varphi)=-\frac12 V_\La(T_a\varphi-\varphi),
\]
which by the argument used in the proof of Theorem~\ref{thmA} shows that the
desired equality for $\rho_{\mathsf \Delta_a \Pi_\La}$ holds outside the origin.
Hence, it suffices to analyze the possible atom at the origin for the spectral measure.
This is deferred to Appendix~\ref{s:atom4}.
\end{proof}

\appendix

\section{The stationary phase bound}
\label{app:stationary-phase}
The purpose of this section is to prove the following simple stationary
phase bound. It is standard, but we did not find a textbook reference for the
version we need.

\begin{proposition}
\label{prop:stat-phase}
There exists a universal constant $C$, such that for any $f\in C^1(\bT)$,
\[
\left|\int_a^b e^{-{\rm i}\omega\cos(\theta)}f(\theta){\rm d}\theta\right| \le
C\,\omega^{-\frac12}\left(\lVert f\rVert_{L^\infty(\bT)}
+\lVert f'\rVert_{L^1(\bT)}\right),
\]
for $\omega\ge 1$ and $0\le a<b\le 2\pi$.
\end{proposition}

\begin{proof}
If the critical points of the phase (i.e. $0$, $\pi$)
are bounded away from the end-points,
we may isolate the end-points with the help
of a cut-off function. The conclusion then
follows by estimating the end-point contributions
with the van der Corput Lemma
(see Proposition~2, Ch. {\rm VIII} in \cite{Stein}),
and the contributions from the interior
of $[a,b]$ with the standard stationary
phase bound for an interior critical point
(see Theorem~7.7.5 in \cite{hormander}).

Thus, it suffices to consider integrals
$\int_a^b e^{-{\rm i}\omega\cos(\theta)}f(\theta){\rm d}\theta$
for $a=0$ and small $b>0$. The change of variables $\theta=\theta(x)$
defined by $\cos(\theta)=1-x^2/2$ is smooth and
bijective as a map $\theta:[0,b']\to [0,b]$, where $b'>0$. Moreover,
$\theta(x)$, $\theta'(x)$ are both uniformly bounded, and $\theta(x)=x+O(x^2)$
uniformly on $[0,b']$ as $b\to 0$. We rewrite
the integral using this change of variables
\[
\int_0^b e^{-{\rm i}\omega\cos(\theta)}f(\theta){\rm d}\theta
=\int_0^{b'}e^{-{\rm i}\omega x^2}f(\theta(x))\theta'(x){\rm d}x=:
\int_0^{b'}e^{-{\rm i}\omega x^2}g(x){\rm d}x
\]
where we have introduced the function $g(x)=f(\theta(x))\theta'(x)$,
and where $b'=b+O(b^2)$ as $b\to 0$.
It is clear from the properties of $\theta(x)$ that
$|g|\le C|f|$ and that $|g'|\le C|f'|$, for some universal constant $C$.

An integration by parts shows that
\begin{equation}
\label{eq:integrate-by-parts}
\int_0^{b'}e^{-{\rm i}\omega x^2}g(x){\rm d}x=
g(b')\int_0^{b'}e^{-{\rm i}\omega x^2}{\rm d}x
- \int_0^{b'}\left(\int_0^{x}e^{-{\rm i}\omega y^2}{\rm d}y\right)g'(x){\rm d}x.
\end{equation}
The Gaussian integral satisfies
\[
\left|\int_0^{b'}e^{-{\rm i}\omega x^2}{\rm d}x\right|=
\frac{\sqrt{\pi}}{\sqrt{2\omega}}\big|\mathrm{erf}\big(b'\sqrt{{\rm i}\omega}\big)\big|
\le C\omega^{-\frac12}
\]
for $\omega,b'>0$, so the first term on the RHS of \eqref{eq:integrate-by-parts} is
bounded above by $C\omega^{-\frac12}\lVert f\rVert$.
The second term can be estimated as follows
\[
\left|\int_0^{b'}\left(\int_0^{x}e^{-{\rm i}\omega y^2}{\rm d}y\right)g'(x){\rm d}x\right|\le
\lVert g'\rVert_{L^1(0,b')}\left|\int_0^{x}e^{-{\rm i}\omega y^2}{\rm d}y\right|\le
C\lVert f'\rVert_{L^1(\bT)}\omega^{-\frac12}.
\]
Hence, the proof is complete.
\end{proof}

\section{Atoms at the origin for the spectral measures}
\label{app:atom}
In this section, we collect all computations concerning the atoms at the origin 
for the various spectral measures from Sections~\ref{s:stat} and \ref{s:potential}.
This will finalize the proofs of Theorems~\ref{thmA}, \ref{thmB} and \ref{thmC}.

\subsection{The atom of \texorpdfstring{$\rho_{{\sf \Delta}_a\zeta_\La}$}{delta-rho-lambda}}
\label{s:atom1}
We will first prove the remaining part of Theorem~\ref{thmA} concerning the
spectral measure of ${\sf \Delta}_a\zeta_\La$, namely that
\[
\rho_{{\sf \Delta}_a\zeta_\La}(\{0\})=\pi^2|a|^2\rho_\La(\{0\}).
\]
To this end, note that by 
the relation \eqref{eq:Deltaa}, translation-boundedness of $\rho_{\La}$, and the
asymptotic formula \eqref{eq:bessel_function_at_infinity} for the Bessel function, the function
$\xi\mapsto \sup_{R\ge 1}J_0(2\pi R |\xi|)$ is square-integrable with
respect to $\rho_{{\sf \Delta}_a\zeta_\La}$. As a consequence
\begin{align}\label{eq:formula-atom-prel}
\nonumber
\rho_{{\sf \Delta}_a\zeta_\La}(\{0\})
&= \int_{\bC}\lim_{R\to \infty} (J_0(2\pi R|\xi|))^2 \, 
{\rm d}\rho_{{\sf \Delta}_a\zeta_\La}(\xi)
\\ &=\lim_{R\to \infty} \int_{\bC}(J_0(2\pi R|\xi|))^2 \, 
{\rm d}\rho_{{\sf \Delta}_a\zeta_\La}(\xi),
\end{align}
where the last step follows from the dominated convergence theorem.
We next claim that the 
right-hand side of \eqref{eq:formula-atom-prel}
may be interpreted as a variance
\begin{equation}\label{eq:var-mean-circle}
 \int_{\bC} (J_0(2\pi R|\xi|))^2 \, 
{\rm d}\rho_{{\sf \Delta}_a\zeta_\La}(\xi)=
{\sf Var}\left[\frac{1}{2{\pi {\rm i}}}
\int_{\{|z|=R\}} \frac{{\sf \Delta}_a\zeta_\La(z)}{z} \, 
{\rm d} z\right]\,.
\end{equation}
To see this, first note that
\[
J_0(2\pi R|\xi|)=\frac{1}{2\pi {\rm i}} 
\int_{\{|z|=R\}}e^{-2\pi {\rm i} \xi \cdot z}
\frac{{\rm d} z}{z}=\widehat{\sigma}_R(\xi)\,,
\]
where $\sigma_R$ denotes the normalized arc-length measure on the circle
$\{|z|=R\}$. By considering a suitable mollifier $h_j(\xi)=j^2 h(j\xi)$ with 
$h\in\mathcal{S}$, $\int_{\bC} h{\rm d}m=1$, and
putting $\varphi_{R,j}=\sigma_R*h_j$, we may write
\[
\frac{1}{2\pi {\rm i}} \int_{\{|z|=R\}} 
\frac{{\sf \Delta}_a\zeta_\La(z)}{z} \, {\rm d} z
=\lim_{j\to \infty}\int_{\bC}{\sf \Delta}_a\zeta_\La(z)
\varphi_{R,j}(z)\,{\rm d}m(z),
\]
where the limit is taken in $L^2(\Omega,\bP)$. 
The formula \eqref{eq:var-mean-circle} then follows
by a reverse application of the Parseval identity 
\begin{align*}
{\sf Var}\Big[\int_{\bC}{\sf \Delta}_a\zeta_\La(z)
\varphi_{R,j}(z)\,{\rm d}m(z)\Big]&
=\int_{\bC}|\widehat{\varphi}_{R,j}(\xi)|^2{\rm d}
\rho_{{\sf \Delta}_a\zeta_\La}(\xi)
\\&=
\int_{\bC}|\widehat{\sigma}_R(\xi)|^2\,
|\widehat{h}_j(\xi)|^2{\rm d}
\rho_{{\sf \Delta}_a\zeta_\La}(\xi)
\end{align*}
(cf. \eqref{eq21} and Remark~\ref{rem:L2-extension}).
Combining \eqref{eq:var-mean-circle} with \eqref{eq:formula-atom-prel}, 
we obtain the representation
\begin{align*}
\rho_{{\sf \Delta}_a\zeta_\La}(\{0\})
&= \lim_{R\to \infty} \int_{\bC} (J_0(2\pi R|\xi|))^2 \, 
{\rm d}\rho_{{\sf \Delta}_a\zeta_\La}(\xi) \\
& = \lim_{R\to \infty}
{\sf Var}\left[\frac{1}{2{\pi {\rm i}}} \int_{|z|=R}
\frac{{\sf \Delta}_a\zeta_\La(z)}{z} \, {\rm d} z\right]
\end{align*}
for the atom at the origin.

Now, recall the formula
\[
\frac{1}{2\pi {\rm i}}\int_{\{|z|=R\}}\frac{1}{z-w}\frac{{\rm d}z}{z}=
\begin{cases}
0,& |w|<R\\
-\frac{1}{w},& |w|>R
\end{cases}
\]
which, since a.s.$\!$ there are no points of $\La$ which
lie on the circle $\{|z|=R\}$, implies that
\[
\frac{1}{2\pi {\rm i}}\int_{\{|z|=R\}}\frac{{\sf \Delta}_a\zeta_\La(z)}{z}{\rm d}z
=\Big[\zeta_\La(a) - \sum_{|\la-a|<R}\frac{1}{a-\la}\Big]
-\Big[\zeta_\La(0)-\sum_{|\la|< R}\frac{1}{-\la}\Big].
\]
But the right-hand side equals
\[
\Big[\zeta_\La(a) - \sum_{|\la-a|<R}\frac{1}{a-\la}-\Psi_1(R)-\pi\mathfrak{c}_\La\bar{a}\Big]
-\Big[\zeta_\La(0)-\sum_{|\la|< R}\frac{1}{-\la}-\Psi_1(R)\Big]+\pi\mathfrak{c}_\La\bar{a}
\]
and by
invoking the representation \eqref{eq:pointwise_L_2_limit_for_zeta} of $\zeta_\La$
we see that the first
two terms on the right-hand side tend to zero in $L^2(\Omega,\bP)$ as $R\to\infty$. Hence, we get
\[
\lim_{R\to\infty}{\sf Var}\Big[\frac{1}{2\pi {\rm i}}
\int_{\{|z|=R\}}\frac{{\sf \Delta}_a\zeta_\La(z)}{z}{\rm d}z\Big]=
\pi^2 |a|^2{\sf Var}[\mathfrak{c}_\La]=\pi^2 |a|^2\rho_\La(\{0\}).
\]
Putting this together with the above representation for
$\rho_{{\sf \Delta}_a\zeta_\La}(\{0\})$, we see that
\[
\rho_{{\sf \Delta}_a\zeta_\La}(\{0\}) = \pi^2|a|^2\rho_\La(\{0\}),
\]
which completes the proof. \qed

\subsection{The atom of \texorpdfstring{$\rho_{\wp_\La}$}{rho-lambda}}
\label{s:atom2}
We next finalize the proof of Theorem~\ref{thmA}, which
amounts to verifying that $\rho_{\wp_\La}(\{0\})=0$.
The starting point is the analogous formula to \eqref{eq:formula-atom-prel}, namely
\[
\rho_{\wp_{\La}}(\{0\})
= \int_{\bC}\lim_{R\to \infty} \left|\frac{J_1(2\pi R|\xi|)}{\pi R|\xi|}\right|^2{\rm d}
\rho_{\wp_\La}(\xi).
\]
We next observe that
\[
\frac{J_1(2\pi R|\xi|)}{\pi R |\xi|}=\frac{1}{\pi R^2}\widehat{\done}_{R\bD}(\xi)\,,
\]
and since the LHS is bounded above by an $L^2(\bC,\rho_{\wp_\La})$-integrable function,
the dominated convergence theorem gives that
\[
\rho_{\wp_{\La}}(\{0\})
=\lim_{R\to\infty}{\sf Var} \left[\wp_{\La}\Big(\frac{1}{\pi R^2}\done_{R\bD}\Big)\right]
\]
Using the familiar formula,
\begin{align*}
\text{p.v.} \int_{R\bD} \frac{1}{(z-\la)^2} \, {\rm d}m(z) = \begin{cases}
0 & \text{for}\ |\la| < R, \\ \frac{\pi R^2}{\la^2} & \text{for} \ |\la|> R,
\end{cases}
\end{align*}
we get the representation
\[
\wp_\La(\done_{R\bD}) = \pi R^2(\Psi_2(\infty) - \Psi_2(R)) \, .
\]
By Lemma~\ref{lem2}, $\wp_{\La}(\done_{R\bD}) \in L^2(\Omega,\bP)$ and
\[
\lim_{R\to \infty } {\sf Var} \left[\wp_{\La}\Big(\frac{1}{\pi R^2}\done_{R\bD}\Big)\right]
= \lim_{R\to\infty} {\sf Var} \left[\Psi_2(\infty) - \Psi_2(R)\right] = 0,
\]
which completes the proof. \qed

\subsection{The atom of \texorpdfstring{$\rho_{V_\La}$}{}}
\label{s:atom3}
We turn to proving the last remaining statement of Theorem~\ref{thmB}, i.e.,
that $\rho_{V_\La}(\{0\})=0$.
Arguing as in Appendix~\ref{s:atom1}, we find that
\begin{align*}
\rho_{V_\La}(\{0\}) &= \int_{\bC}\lim_{R\to \infty} (J_0(2\pi R|\xi|))^2 \,
{\rm d}\rho_{V_\La}(\xi) \\
&= \lim_{R\to \infty} \int_{\bC} (J_0(2\pi R|\xi|))^2 \,
{\rm d}\rho_{V_\La}(\xi) \\
& = \lim_{R\to \infty} {\sf Var}\left[\frac{1}{2{\pi {\rm i}}}
\int_{\{|z|=R\}} \frac{V_\La(z)}{z} \, {\rm d} z\right].
\end{align*}
For any given $R\ge 1$, a.s.$\!$ , there are no points of $\La$ which
lie on the circle $\{|z|=R\}$. By a residue computation, we get
\begin{align*}
\frac{1}{2{\pi {\rm i}}} \int_{\{|z|=R\}} \frac{V_\La(z)}{z} \, {\rm d} z
&= \frac{1}{2{\pi {\rm i}}}
\int_{\{|z|=R\}} \frac{\zeta_\La(z)-\Psi_1(\infty)}{z} \, {\rm d} z \\
&= \zeta_\La(0)+\sum_{|\la|\le R}\frac{1}{\la}-\Psi_1(\infty) \\ &
= -\sum_{|\la|< 1} \frac{1}{\la} +\sum_{|\la|<R} \frac{1}{\la}-\Psi_1(\infty)
= \Psi_1(R)-\Psi_1(\infty)\, .
\end{align*}
But Lemma~\ref{lem3} implies that
\[
{\sf Var}\big[\Psi_1(R)-\Psi_1(\infty)\big]=0
\]
which completes the proof. \qed

\subsection{The atom of \texorpdfstring{$\rho_{{\sf \Delta}_a\Pi_\La}$}{}}
\label{s:atom4}
To conclude the proof of Theorem~\ref{thmC}, it only remains to prove that
\[
\rho_{{\sf \Delta}_a\Pi_\La}(\{0\})=\frac{\pi^2|a|^4}{4}\rho_\La(\{0\}).
\]
We again use the method introduced in Appendix~\ref{s:atom1},
in particular the formula
\[
\rho_{{\sf \Delta}_a\Pi_\La}\left(\{0\}\right)=
\lim_{R\to \infty} \int_{\bC} (J_0(2\pi R|\xi|))^2 \,
{\rm d}\rho_{{\sf \Delta}_a\Pi_\La}(\xi)
=
\lim_{R\to\infty}
{\sf Var}\Big[\frac{1}{2\pi{\rm i}}
\int_{\{|z|=R\}}{\sf \Delta}_a\Pi_\La(z)\frac{{\rm d}z}{z}\Big].
\]
We start with the representation
\[
{\sf \Delta}_a\Pi_\La(z)=\lim_{S\to\infty}
\sum_{|\la|\le S}\big(\log|z-(\la-a)|-\log|z-\la|\big)
-\frac12\pi\mathfrak{c}_\La\big(2\Re \bar{a}z)-|a|^2\big)
\]
and the formula
\[
\frac{1}{2\pi R}\int_{\{|z|=R\}}\log|z-w|\,|{\rm d}z|
=\begin{cases}\log R&|w|\le R\\ \log|w|&|w|>R\end{cases}.
\]
Combining these two, we get
\begin{multline*}
\frac{1}{2\pi {\rm i}}\int_{\{|z|=R\}}{\sf \Delta}_a\Pi_\La(z)\frac{{\rm d}z}{z}
=\lim_{S\to\infty}\left[\sum_{|\la|\le S,\;|\la-a|\ge R}\log|\la-a|
-\sum_{R\le |\la|\le S}\log|\la|\right]\\
+\sum_{|\la-a|\le R}\log R-\sum_{|\la|\le R}\log R
 -\frac12\pi\mathfrak{c}_\La|a|^2,
\end{multline*}
where we have used the fact that $\Re(\bar{a}z)$ has zero
mean over circles centered at the origin.
By subtracting and adding back $\sum_{|\la-a|\ge R,\,|\la|\le S}\log|\la|$,
we may rewrite the right-hand side as
\[
\sum_{|\la-a|\ge R}\big(\log|\la-a|-\log|\la|\big)
+\left[\sum_{|\la|\le R}\log\frac{|\la|}{R}
-\sum_{|\la-a|\le R}\log\frac{|\la|}{R}\right]-\frac12\pi\mathfrak{c}_\La|a|^2.
\]
In view of the spectral condition,
Lemma~\ref{lem3} on the convergence of the sequence
$(\Psi_1(R))_{R\ge 1}$ gives that the
first term tends to zero in $L^2(\Omega,\bP)$.
To obtain the result, if suffices to show that the variance of
\[
{\sf Var}\left[\sum_{|\la|\le R}\log\frac{|\la|}{R}
-\sum_{|\la-a|\le R}\log\frac{|\la|}{R}\right]=
{\sf Var}\left[\sum_{\la\in A}\log\frac{|\la|}{R}
-\sum_{\la\in B}\log\frac{|\la|}{R}\right]
\]
tends to zero as $R\to\infty$,
where we use the notation $A$ and $B$ for the lunar domains
\[
A=\{|\la|\le R,\;|\la-a|\ge R\}\quad\text{and}\quad B
=\{|\la|\ge R,\;|\la-a|\le R\}.
\]
The corresponding linear statistic is
\[
f_R(\la)=\big(\done_{A}(\la)-\done_{B}(\la)\big)
\log\frac{|\la|}{R}=:f_{R,A}(\xi)-f_{R,B}(\xi),
\]
so the variance equals
\[
{\sf Var}\left[\sum_{\la\in A}\log\frac{|\la|}{R}
-\sum_{\la\in B}\log\frac{|\la|}{R}\right]
=\int_{\bC\setminus\{0\}}
|\widehat{f}_R(\xi)|^2{\rm d}\rho_\La(\xi),
\]
where we have used the fact that $\widehat{f}_{R}(0)=0$ to infer
that the origin does not contribute to the variance.
We need to estimate $|\widehat{f}_R(\xi)|^2$ from above,
and by symmetry it suffices
to consider only $f_{R,A}$.

Since $\log\frac{|\la|}{R}=O(R^{-1})$ on $A$, the naive bound
$|\widehat{f}_{R,A}(\xi)|\le \lVert f_{R,A}\rVert_{L^1(\bC)}$
gives that $|\widehat{f}_{R,A}(\xi)|^2=O(1)$ as $R\to\infty$ for $\xi\in\bC$.
In order to complete the proof, we need better estimates valid for large $|\xi|$.
An application of Green's formula shows that
\begin{align*}
\widehat{f}_{R,A}(\xi)&
=\int_A\log\frac{|\la|}{R}e^{-2\pi{\rm i} \xi\cdot \la}{\rm d}m(\la)\\
&=\frac{1}{\pi^2|\xi|^2}\int_A\log\frac{|\la|}{R}
\Delta\left(e^{-2\pi{\rm i}\xi\cdot \la}\right){\rm d}m(\la)\\
&=\frac{1}{\pi^2|\xi|^2}\int_{\partial A}e^{-2\pi{\rm i}\xi\cdot \la}
\left[-2\pi{\rm i}\log\frac{|\la|}{R}\big\langle \xi, N(\la)\big\rangle
+\left\langle \tfrac1{\bar\la},N(\la)
\right\rangle\right]|{\rm d}\la|,
\end{align*}
where $N$ denotes the outward unit normal to $A$. Parameterizing the two
circular arcs that make up the boundary $\partial A$,
we find that
\[
\widehat{f}_{R,A}(\xi)=\frac{1}{\pi^2|\xi|^2}\big(J_1(\xi,R,a)+J_2(\xi,R,a)\big),
\]
where $J_1$ and $J_2$ are the integrals
\begin{equation}
\label{eq:param-1}
J_1(R,\xi,a)=\frac{1}{R}\int_{s_1}^{t_1}e^{-2\pi{\rm i}R|\xi|\cos(\theta)}{\rm d}\theta
\end{equation}
and
\begin{align}
\label{eq:param-2}
\begin{split}
J_2(R,\xi,a)=&-2\pi {\rm i}|\xi|\int_{s_2}^{t_2}e^{-2\pi{\rm i}R|\xi|\cos(\theta)}
\log\big|1+\tfrac{ae^{-{\rm i}\theta}}{R}\big|\cos(\theta){\rm d}\theta
\\&+\frac{1}{R}\int_{s_2}^{t_2}e^{-2\pi{\rm i}R|\xi|\cos(\theta)}
\Re\Big[\tfrac{1}{1+R^{-1}\bar{a}e^{-{\rm i}\theta}}\Big]{\rm d}\theta
\end{split}
\end{align}
respectively. Here, $s_i$ and $t_i$ are positive parameters with $0\le s_i\le t_i\le 2\pi$,
depending on $R$ and $a$ and $\xi$.
For large $|\xi|$, we may thus use the standard stationary phase bound
\[
\left|\int_{\alpha}^\beta e^{-2\pi {\rm i}\omega \cos(\theta)}f(\theta){\rm d}\theta\right|
\lesssim \omega^{-\frac12}\big(\lVert f\rVert_{L^\infty([\alpha,\beta])}+
\lVert f'\rVert_{L^1([\alpha,\beta])}\big)
\]
from Proposition~\ref{prop:stat-phase} in the Appendix
to obtain the estimates
\[
\big|J_1(R,\xi,a)\big|\lesssim |\xi|^{-\frac12}R^{-\frac32}
\]
and
\[
\big|J_1(R,\xi,a)\big|\lesssim_a |\xi|^\frac12R^{-\frac32}+\xi^{-\frac12}R^{-\frac32}
\]
for the two contributions $J_1(R,\xi,a)$ and $J_2(R,\xi,a)$ for $\widehat{f}_{R,A}(\xi)$.
Putting everything together, we arrive at
\[
|\widehat{f}_{R,A}(\xi)|^2=\frac{1}{\pi^4|\xi|^4}
\big|J_1(\xi,R,a)+J_2(\xi,R,a)\big|^2\lesssim_a
\min\big\{1,R^{-3}|\xi|^{-3}\big\}.
\]
Since the corresponding bound holds also for $\widehat{f}_{R,B}$, we get
\[
{\sf Var}\left[n_\La(f_R)\right]=\int_{\bC\setminus\{0\}}
\big|\widehat{f}_R(\xi)\big|^2{\rm d}\rho_\La(\xi)
\lesssim_a \int_{\bC\setminus\{0\}}
\min\big\{1,R^{-3}|\xi|^{-3}\big\}{\rm d}\rho_\La(\xi),
\]
and hence the result follows by invoking the dominated convergence theorem.
\qed

\subsection*{Acknowledgements}
We would like to thank Fedor Nazarov, Alon Nishry and Ron Peled for
several very useful discussions throughout the work on this project.
We also thank the 
anonymous referees for numerous insightful suggestions.

During part of this work, the second author was based at Tel Aviv University.
He would like to express his gratitude for the excellent scientific environment provided there.

The research of M.S.\ and O.Y.\ was supported by
ERC Advanced Grant 692616, ISF Grant 1288/21 and
by BSF Grant 202019. The research of A.W.\ was
supported by the KAW foundation Grant 2017.0398, by
ERC Advanced Grant 692616 and by Grant No. 2022-03611 
from the Swedish Research Council (VR).

\subsection*{Data availability statement}
Data sharing is not applicable to this article as no datasets 
were generated or analysed during the current study.

\subsection*{Compliance with ethical standards}
The authors do not have any potential conflicts of interests to disclose.

\bigskip
\medskip

\noindent Sodin:
School of Mathematical Sciences, Tel Aviv University, Tel Aviv, Israel
\newline {\tt sodin@tauex.tau.ac.il}
\smallskip\newline\noindent{Wennman: Department of Mathematics, KTH Royal Institute of Technology, 
Stockholm, Sweden
\newline {\tt aronw@kth.se}
\smallskip\newline\noindent Yakir:
School of Mathematical Sciences, Tel Aviv University, Tel Aviv, Israel
\newline {\tt oren.yakir@gmail.com}
}


\begin{thebibliography}{A}

\bibitem{AGL} M. Aizenman, S. Goldstein, J. Lebowitz,
Bounded fluctuations and translation symmetry breaking in
one-dimensional particle systems.
Special issue dedicated to the memory of Joaquin M. Luttinger,
J. Stat. Phys. {\bf 103} (2001), 601--618.

\bibitem{Al-Janc} A. Alastuey, B. Jancovici,
On potential and field fluctuations in two-dimensional classical charged systems.
J. Stat. Phys. {\bf 34} (1984), 557--569.

\bibitem{Becker} M. Becker,
Multiparameter groups of measure-preserving transformations: a
simple proof of {W}iener's ergodic theorem.
Ann. Probab. {\bf 9} (1981) 504--509.

\bibitem{BuckleySodin} J. Buckley, M. Sodin,
Fluctuations of the increment of the argument for the {G}aussian entire function.
J. Stat. Phys. {\bf 168} (2017), 300--330.


\bibitem{BGLS} L. Buhovsky, A. Gl\"ucksam, A. Logunov, M. Sodin,
Translation-invariant probability measures on entire functions.
J. Anal. Math. {\bf 139} (2019), 307--339.

\bibitem{Chandra} S. Chandrasekhar,
Stochastic problems in physics and astronomy.
Rev. Modern Phys. {\bf 15} (1943), 1--89.

\bibitem{CPPR} S. Chatterjee, R. Peled, Y. Peres, D. Romik,
Gravitational allocation to {P}oisson points.
Ann. of Math. {\bf 172} (2010), 617--671.

\bibitem{DV-J} D. J. Daley, D. Vere--Jones,
An introduction to the theory of point processes. Vol. I.
Elementary theory and methods. Second edition.
Probability and its Applications (New York).
Springer-Verlag, New York, 2003. xxii+469 pp.


\bibitem{DV-J2} D. J. Daley, D. Vere--Jones,
An introduction to the theory of point processes. Vol. II.
General Theory and Structure. Second edition.
Probability and its Applications (New York).
Springer-Verlag, New York, 2003. xvii+573 pp.

\bibitem{Dav} Yu. S. Davidovich,
The asymptotic behavior of dispersions of space means of a homogeneous random field.
Prob. Th. and Math. Stat. {\bf 3} (1970), 35--49 (in Russian).

\bibitem{FH} P. Forrester, G. Honner,
Exact statistical properties of the zeros of complex random polynomials.
J. Phys. A {\bf 32} (1999), 2961--2981.

\bibitem{GhoshLebowitzRigidity} S. Ghosh, J. L. Lebowitz,
Number Rigidity in Superhomogeneous Random Point Fields. J. Stat Phys.
(2017), \textbf{166}, 1016--1027.


\bibitem{GV} I. M. Gelfand, N. Ya. Vilenkin,
Generalized functions. {V}ol. 4, Applications of harmonic analysis.
AMS Chelsea Publishing, Providence, 2016.

\bibitem{GP} S. Ghosh, Y. Peres,
Rigidity and tolerance in point processes: {G}aussian zeros and {G}inibre eigenvalues.
Duke Math. J. {\bf 166} (2017), 1789--1858.

\bibitem{hormander} L. H\"{o}rmander,
The analysis of linear partial differential operators. {I}.
Second edition. Springer-Verlag, Berlin, 1990.

\bibitem{IbrLinnik} I. Ibragimov, Yu. Linnik,
Independent and stationary sequences of random variables.
Translation from the Russian edited by J. F. C. Kingman,
Wolters-Noordhoff Publishing, Groningen, 1971.

\bibitem{HKPV}
J. Ben Hough, M. Krishnapur, Y. Peres, B. Vir\'ag,
Zeros of {G}aussian analytic functions and determinantal point processes.
American Mathematical Society, Providence, RI, 2009.

\bibitem{Leb-Martin} J. Lebowitz, Ph. Martin,
On potential and field fluctuations in classical charged systems.
J. Stat. Phys. {\bf 34} (1984), 287--311.

\bibitem{Leonov} V. P. Leonov,
On the dispersion of time means of a stationary stochastic process.
Prob. Th. and Its Appl. {\bf 6} (1961), 93--101 (in Russian)

\bibitem{Lewin} M. Lewin,
Coulomb and Riesz gases: The known and the unknown.
J. Math. Phys. 63 (2022), no. 6, Paper No. 061101, 77 pp.

\bibitem{Martin} Ph. A. Martin, Sum rules in charged fluids.
Rev. Modern Phys. 60 (1988), no. 4, 1075–1127.

\bibitem{NS-WhatIs} F. Nazarov, M. Sodin,
What is{$\ldots$}a {G}aussian entire function?
Notices Amer. Math. Soc. {\bf 57} (2010), 375--377.

\bibitem{NS-IMRN} F. Nazarov, M. Sodin,
Fluctuations in random complex zeroes: asymptotic normality revisited.
Int. Math. Res. Not. IMRN {\bf 24} (2011), 5720--5759.

\bibitem{Robinson} E. A. Robinson,
Sums of stationary random variables.
Proc. Amer. Math. Soc. {\bf 11} (1960), 77--79.

\bibitem{Schmidt} K. Schmidt,
Cocycles on ergodic transformation groups.
Macmillan Company of India, Ltd., Delhi, 1977.

\bibitem{SerfatySurv1} S. Serfaty,
Systems of points with Coulomb interactions. Proceedings of the International Congress 
of Mathematicians — Rio de Janeiro 2018. Vol. I. Plenary lectures, 935–977, 
World Sci. Publ., Hackensack, NJ, 2018.

\bibitem{SerfatySurv2}
S. Serfaty, Microscopic description of Log and Coulomb gases. Random matrices, 341–387,
IAS/Park City Math. Ser., 26, Amer. Math. Soc., Providence, RI, 2019.

\bibitem{SWY-II} M. Sodin, A. Wennman, O. Yakir,
The random Weierstrass zeta function II.
Fluctuations of the electric flux through rectifiable curves.
arXiv:2211.01312 (2022).

\bibitem{Stein} E. M. Stein. Harmonic analysis:
real-variable methods, orthogonality, and oscillatory integrals.
With the assistance of Timothy S. Murphy. Princeton Mathematical Series, 43.
Monographs in Harmonic Analysis, III. {\em Princeton University Press, Princeton, NJ},
1993. xiv+695 pp.

\bibitem{Taylor} M. Taylor, The Weierstrass $\wp$-function as a Distribution on the
Complex Torus $\bC\setminus\La$, and its Fourier Series.

\bibitem{Vergara} R. C. Vergara,  D. Allard, N. Desassis,
A general framework for SPDE-based stationary random fields.
Bernoulli \textbf{28} no.\ 1 (2022), 1--32.

\bibitem{watson} G. N. Watson,
A treatise on the theory of {B}essel functions. Second edition.
Cambridge University Press, Cambridge, 1995.

\bibitem{Weiss} B. Weiss,
Measurable entire functions.
The heritage of P. L. Chebyshev: a Festschrift in honor of the
70th birthday of T. J. Rivlin, Ann. Numer. Math. {\bf 4} (1997), 599--605.

\bibitem{Yaglom} A. M. Yaglom,
Some classes of random fields in $n$-dimensional space, related to stationary random processes.
Prob. Th. and Its Appl. {\bf 2} (1957), 273--320.

\bibitem{Yakir} O. Yakir,
Fluctuations of linear statistics for {G}aussian perturbations
of the lattice {$\mathbb{Z}^d$}.
J. Stat. Phys. {\bf 182} (2021), Paper No. 58, 21 pp.
\end{thebibliography}
\end{document}